\documentclass[article]{my-mesa-NS}
\usepackage{times,mathptm}
\usepackage{amsmath,amsfonts,amssymb}
\usepackage{graphicx,graphics,epsfig}
\usepackage{fancyhdr,fancybox}
\usepackage{verbatim}
\usepackage{color}

\jvol{20}
\jnum{4}
\jyear{2013}

\begin{document}
\label{pageinit}

\date{}


\title{Uniform approximation of fractional derivatives\\
and integrals with application to fractional\\
differential equations}

\author{Hassan Khosravian-Arab$^{1}$, Delfim F. M. Torres$^{2,\star}$}

\maketitle


\noindent $^1$ Department of Applied Mathematics,
Faculty of Mathematics and Computer Science,\\
Amirkabir University of Technology,
Hafez Ave, Tehran, Iran.\\
\noindent $^2$ CIDMA --- Center for Research and Development in Mathematics and Applications,\\
Department of Mathematics, University of Aveiro, 3810-193 Aveiro, Portugal.\\

\noindent $^\star$ \emph{Corresponding Author}. E-mail: delfim@ua.pt

\markboth{H. Khosravian-Arab and D. F. M. Torres}{Uniform
approximation of fractional derivatives and integrals}



\begin{abstract}
It is well known that for every $f\in C^m$ there exists a polynomial
$p_n$ such that $p^{(k)}_n\rightarrow f^{(k)}$, $k=0,\ldots,m$.
Here we prove such a result for fractional (non-integer)
derivatives. Moreover, a numerical method is proposed
for fractional differential equations. The convergence rate and stability
of the proposed method are obtained. Illustrative examples are discussed.
\end{abstract}


\section{Introduction}

The analysis and design of many physical systems require
the solution of a fractional differential equation
\cite{book:Benchohra,KilSriTru:06:The}.
Examples include fractional oscillation equations
\cite{LinLiu:07:Fra,Pod:99:Fra(b)},
linear and nonlinear fractional Bagley--Torvik equations
\cite{El-El-El:05:Num,Yua:09:Sol}, Basset equations
\cite{EdwForSim:02:The}, fractional
Lorenz systems \cite{AloNooNaz:09:Hom},
and population models with fractional order derivatives
\cite{MyID:215,Xu:09:Ana}.

Several methods have recently been proposed to address
both linear and nonlinear \emph{Fractional Ordinary Differential Equations} (FODEs):
an analytical method to solve linear FODEs is given in \cite{HuLuoLu:08:Ana},
homotopy methods for the analysis of linear and nonlinear FODEs are used in
\cite{HosMil:09:An,HasAbdMom:09:Hom}, Adams multistep methods for nonlinear FODEs
are investigated in \cite{Gar:09:On}, and a differential transform method
is developed in \cite{AytIbr:07:Sol}.
Another widely used procedure consists in transforming
a differential equation with fractional derivatives
into a system of differential equations of integer order
\cite{MyID:225}. Other numerical schemes to solve FODEs can be found in
\cite{ForCon:06:Com,SaaDeh:09:Fra}.

Here we propose the use of Bernstein polynomials to approximate
fractional derivatives and integrals.
Numerical algorithms based on Gr\"{u}nwald and modified Gr\"{u}nwald
approximation formulas were proposed and analyzed
in \cite{Delb:Elliott:94,Elliott:93}
for the approximate evaluation of certain Hadamard integrals,
where Bernstein polynomials are used to derive an error bound
\cite{Delb:Elliott:94,Elliott:93}.
Recently, Bernstein polynomials have been used to approximate the
solution of fractional integro-differential equations \cite{Mohammed:12},
fractional heat- and wave-like equations \cite{Rostamy:12},
and multi-dimensional fractional optimal control problems \cite{Ali_Rost:Bal:JVC}.
For more on approximation of fractional derivatives and its applications,
we refer the reader to \cite{MyID:258} and references therein.

Despite numerous results available in the literature,
there is still a significant demand for readily usable numerically algorithms
to handle mathematical problems involving fractional derivatives and
integrals. To date, such algorithms have been developed but only to a rather
limited extent. Here we specify, develop and analyze,
a general numerical algorithm and present such scheme
in a rigorous but accessible way, understandable to an applied scientist.

The paper is organized as follows.
Section~\ref{sec:2} reviews, briefly, the necessary definitions
and results concerning Bernstein polynomials.
We also recall the standard definitions of fractional integral
and fractional derivative in the sense of Riemann--Liouville and Caputo.
In Section~\ref{sec:3} we introduce the basic idea of our method.
Uniform approximation formulas for fractional derivatives
and integrals are proposed. Some useful results,
and the convergence of the approximation formulas, are proved.
In Section~\ref{sec:4} we analyze the results obtained
in Section~\ref{sec:3} computationally. Finally,
in Section~\ref{sec:5} we apply our approximation formula
for fractional derivatives to the study of a FODE.
We begin by proving the convergence and stability of the proposed numerical method.
Then, some concrete examples are given. The examples considered
show that our method is effective for solving both linear
and nonlinear FODEs in a computationally efficient way.


\section{Preliminaries}
\label{sec:2}

The Weierstrass approximation theorem is a central result of mathematics.
It states that every continuous function defined on a closed interval $[a,b]$
can be uniformly approximated by a polynomial function.

\begin{theorem}[The Weierstrass approximation theorem]
Let $f\in C[a,b]$. For any $\varepsilon>0$,
there exists a polynomial $p_n$ such that
$|f(x)-p_n(x)|\leq\varepsilon$ for all $x\in[a,b]$.
\end{theorem}

There are many proofs to the Weierstrass theorem.
One of the most elegant proofs
uses Bernstein polynomials.

\begin{definition}
\label{def1}
Let $f$ be continuous on $[0,1]$. The Bernstein
polynomial of degree $n$ with respect to $f$ is defined as
\begin{equation}
\label{eq:bp}
B_n(f;x)=\sum_{i=0}^{n}\binom{n}{i}f\left(\frac{i}{n}\right)x^i(1-x)^{n-i}.
\end{equation}
\end{definition}

The next theorem is due to Bernstein,
and provides a proof to the Weierstrass theorem.
Bernstein's Theorem~\ref{thm1} not only proves the existence of
polynomials of uniform approximation, but also provides a simple
explicit representation for them.

\begin{theorem}[Bernstein's theorem \cite{Dav,phi}]
\label{thm1}
Let $f$ be bounded on $[0,1]$. Then,
$\lim_{n\rightarrow\infty}B_n(f;x)=f(x)$
at any point $x\in[0,1]$ at which $f$ is continuous. Moreover,
if $f\in C[0,1]$, then the limit holds uniformly in $[0,1]$.
\end{theorem}

\begin{corollary}
If $f\in C[0,1]$ and $\varepsilon>0$, then
one has, for all sufficiently large $n$, that
$|f(x)-B_n(f;x)|\leq\varepsilon$ for all $0\leq x\leq 1$.
\end{corollary}

In case of functions that are twice differentiable,
an asymptotic error term for the Bernstein polynomials
is easily obtained.

\begin{theorem}[Voronovskaya's theorem \cite{Dav,phi}]
\label{thm11}
Let $f$ be bounded on $[0, 1]$. For any $x\in[0, 1]$
at which $f''(x)$ exists,
$\lim_{n\rightarrow\infty}[2n(B_n(f;x)-f(x))]=x(1-x)f''(x)$.
\end{theorem}

In contrast with other methods of approximation, the Bernstein
polynomials yield smooth approximations. If the approximated
function is differentiable, not only do we have
$B_n(f;x)\rightarrow f(x)$ but also $B'_n(f;x)\rightarrow f'(x)$.
A corresponding statement is true for higher derivatives. Therefore,
Bernstein polynomials  provide simultaneous approximation
of the function and its derivatives.

\begin{theorem}[\cite{Dav,phi}]
\label{thm2}
If $f\in C^p[0,1]$, then
$\lim_{n\rightarrow\infty}B^{(p)}_n(f;x)=f^{(p)}(x)$
uniformly on $[0,1]$.
\end{theorem}

\begin{remark}
All results formulated above for $[0,1]$ are easily transformed to $[a,b]$
by means of the linear transformation
$y=\frac{x-a}{b-a}$ that converts $[a,b]$ into $[0,1]$.
\end{remark}

In Section~\ref{sec:3}, we are going to develop the result presented
in Theorem~\ref{thm2} for fractional derivatives.
For an analytic function $f$ over the
interval $[a,b]$, the left and right sided fractional integrals in the
Riemann--Liouville sense, ${}_{a}\textrm{I}_{x}^{\alpha}f(x)$
and ${}_{x}\textrm{I}_{b}^{\alpha}f(x)$, respectively, are defined by
\begin{equation*}
_{a}\textrm{I}_{x}^{\alpha}f(x):=\frac{1}{\Gamma(\alpha)}
\int_{a}^{x}\frac{f(\tau)}{(x-\tau)^{1-\alpha}}d\tau,
\quad x\in[a,b],
\end{equation*}
and
\begin{equation*}
_{x}\textrm{I}_{b}^{\alpha}f(x):=\frac{1}{\Gamma(\alpha)}
\int_{x}^{b}\frac{f(\tau)}{(\tau-x)^{1-\alpha}}d\tau,
\quad x\in[a,b],
\end{equation*}
where $\alpha>0$ is a real number. For $m-1\leq\alpha< m$,
the left and right sided Riemann--Liouville
fractional derivatives are defined by
\begin{equation*}
_{a}\textrm{D}_{x}^{\alpha}f(x):=\frac{1}{\Gamma(m-\alpha)}\frac{d^m}{dx^m}
\int_{a}^{x}\frac{f(\tau)}{(x-\tau)^{\alpha+1-m}}d\tau,
\quad x\in[a,b],
\end{equation*}
and
\begin{equation*}
_{x}\textrm{D}_{b}^{\alpha}f(x):=\frac{(-1)^m}{\Gamma(m-\alpha)}\frac{d^m}{dx^m}
\int_{x}^{b}\frac{f(\tau)}{(\tau-x)^{\alpha+1-m}}d\tau,
\quad x\in[a,b],
\end{equation*}
respectively. The left and right sided Caputo fractional derivatives are defined by
\begin{equation*}
{}^{C}_{a}\textrm{D}_{x}^{\alpha}f(x):=\frac{1}{\Gamma(m-\alpha)}
\int_{a}^{x}\frac{f^{(m)}(\tau)}{(x-\tau)^{\alpha+1-m}}d\tau,
\quad x\in[a,b],
\end{equation*}
and
\begin{equation*}
{}^{C}_{x}\textrm{D}_{b}^{\alpha}f(x):=\frac{(-1)^m}{\Gamma(m-\alpha)}
\int_{x}^{b}\frac{f^{(n)}(\tau)}{(\tau-x)^{\alpha+1-m}}d\tau,
\quad x\in[a,b],
\end{equation*}
respectively, where $m$ is the integer such that
$m-1\leq\alpha< m$. The next theorem gives a relation
between Caputo and Riemann--Liouville derivatives.

\begin{theorem}[\cite{Pod:99:Fra(b)}]
\label{Thm.1}
If $f\in C^m[0,1]$ and $\alpha \in [m-1,m)$, then
\begin{gather*}
{}^{C}_{0}D_{x}^{\alpha}f(x)={}_{0}D_{x}^{\alpha}f(x)
-\sum_{k=0}^{m-1}\frac{f^{(k)}(0)(x)^{k-\alpha}}{\Gamma(k+1-\alpha)},\\
{}^{C}_{x}D_{1}^{\alpha}f(x)={}_{x}D_{1}^{\alpha}f(x)
-\sum_{k=0}^{m-1}\frac{f^{(k)}(1)(1-x)^{k-\alpha}}{\Gamma(k+1-\alpha)}.
\end{gather*}
\end{theorem}

For more on fractional calculus see, e.g., \cite{KilSriTru:06:The,Pod:99:Fra(b)}.


\section{Fractional derivatives and integrals of the Bernstein polynomials}
\label{sec:3}

We begin by computing the left and right sided
fractional derivatives and integrals of the Bernstein
polynomials with respect to a function $f$. Then, we prove
that Theorem~\ref{thm2} can be formulated for
non-integer derivatives and integrals.


\subsection{Left and right sided Caputo fractional derivatives of $\mathbf{B_n(f;x)}$}

The Bernstein polynomials \eqref{eq:bp} can be written in the form
\[
B_n(f;x)=\sum_{i=0}^{n}\sum_{j=0}^{n-i}\binom{n}{i}\binom{n-i}{j}(-1)^jf\left(\frac
i n\right)x^{i+j}.
\]
The left sided Riemann--Liouville fractional derivative of $B_n(f;x)$ on $[0,1]$ is
\begin{equation}
\label{pre1}
\begin{split}
_{0}D_{x}^{\alpha}B_n(f;x)
&= {_{0}D_{x}^{\alpha}} \sum_{i=0}^{n}\sum_{j=0}^{n-i}\binom{n}{i}
\binom{n-i}{j}(-1)^jf\left(\frac i n\right)x^{i+j}\\
&= \sum_{i=0}^{n}\sum_{j=0}^{n-i}\binom{n}{i}\binom{n-i}{j}f\left(\frac i n\right)
\frac{(-1)^j\Gamma(i+j+1)}{\Gamma(i+j+1-\alpha)}x^{i+j-\alpha}.
\end{split}
\end{equation}
Another representation of the Bernstein polynomials with respect
to function $f$ is obtained when we use the binomial
expansion of $x^i=\left(1-(1-x)\right)^i$:
\begin{equation}
\label{eq:bp:f3}
B_n(f;x)=\sum_{i=0}^{n}\sum_{j=0}^{i}\binom{n}{i}\binom{i}{j}(-1)^jf\left(\frac
i n\right)(1-x)^{n-i+j}.
\end{equation}
Using \eqref{eq:bp:f3}, we get the right sided Riemann--Liouville fractional derivative
of $B_n(f;x)$ over the interval $[0,1]$ as
\begin{equation}
\label{pre2}
\begin{split}
_{x}D_{1}^{\alpha}B_n(f;x)
&= {_{x}D_{1}^{\alpha}}\sum_{i=0}^{n}\sum_{j=0}^{i}
\binom{n}{i}\binom{i}{j}(-1)^jf\left(\frac i n\right)(1-x)^{n-i+j}\\
&= \sum_{i=0}^{n}\sum_{j=0}^{i}\binom{n}{i}\binom{i}{j}
f\left(\frac i n\right)\frac{(-1)^j\Gamma(n-i+j+1)}{\Gamma(n-i+j+1-\alpha)}(1-x)^{n-i+j-\alpha}.
\end{split}
\end{equation}
The left and right sided Caputo fractional derivatives
of the Bernstein polynomials with respect
to a function $f$ on $[0,1]$ are easily obtained
from Theorem~\ref{Thm.1}.

The next theorem shows that
${}_{0}^CD_{x}^{\alpha}B_{n}(f;x)\rightarrow
{}_{0}^CD_{x}^{\alpha}f(x)$ as $n\rightarrow\infty$ uniformly on
$[0,1]$. The same result holds for the right derivative:
${}_{x}^CD_{1}^{\alpha}B_{n}(f;x)\rightarrow{}_{x}^CD_{1}^{\alpha}f(x)$
as $n\rightarrow\infty$ uniformly on $[0,1]$. Along the text,
we use $\| \cdot \|$ to denote the uniform norm over the interval $[0,1]$,
that is, $\|g\|=\max_{0\leq x\leq1}|g(x)|$ for $g\in C[0,1]$.

\begin{theorem}
\label{MThCa}
Let $\alpha$ be a nonnegative real number and
$m \in \mathbb{N}$ be such that $m-1\leq\alpha< m$.
If $f\in C^m[0,1]$ and $\varepsilon>0$, then
$\|{}_{0}^CD_{x}^{\alpha}f-{}_{0}^CD_{x}^{\alpha}B_{n}(f)\|<\varepsilon$.
\end{theorem}

\begin{proof}
We wish to show that, given $\varepsilon>0$ and $f \in C^m[0,1]$,
there exists an integer $n>m$ such that
$\|{}_{0}^CD_{x}^{\alpha}f - {}_{0}^CD_{x}^{\alpha}B_{n}(f)\|<\varepsilon$,
where $B_{n}(f)$ is the Bernstein polynomial of degree $n$ with respect to
function $f$. Using Theorem~\ref{thm2}, we have
$\|f^{(m)}-B^{(m)}_n(f)\|<\epsilon$. Thus,
\begin{equation*}
\begin{split}
\bigl| {_{0}^{C}D_{x}^{\alpha}}f(x)
&-{}_{0}^CD_{x}^{\alpha}B_{n}(f;x)\bigr|\\
&= \left|\frac{1}{\Gamma(m-\alpha)}\left(
\int_0^x(x-t)^{m-\alpha-1}\frac{d^m}{dt^m}(f(t)-B_n(f;t))\,dt\right)\right|\\
&\leq\frac{1}{\Gamma(m-\alpha)}\left(
\int_0^x(x-t)^{m-\alpha-1}\left|f^{(m)}(t)-B^{(m)}_n(f;t)\right|\,dt\right)\\
&\leq\frac{1}{\Gamma(m-\alpha)}\left(
\int_0^x(x-t)^{m-\alpha-1}\|f^{(m)}-B^{(m)}_n(f)\|\,dt\right)\\
&=\epsilon\ \frac{x^{m-\alpha}}{\Gamma(m-\alpha+1)}<\varepsilon.
\end{split}
\end{equation*}
\end{proof}

\begin{remark}
Theorem~\ref{MThCa} is a generalization of Theorem~\ref{thm2}:
if $\alpha \in\Bbb{N}$, then Theorem~\ref{MThCa} reduces to Theorem~\ref{thm2}.
\end{remark}

The next theorem gives an asymptotic error term for the
Caputo fractional derivatives of the Bernstein polynomials with
respect to functions $f\in AC^{m+2}[0,1]$, where
$AC^{m+2}[0,1]$ denotes the space of real functions $f$ that
have continuous derivatives up to order $m+1$ with
$f^{(m+1)}$ absolutely continuous on $[0, 1]$. The
theorem shows that if $\alpha \in [m-1, m)$, then
$|_{0}^CD_{x}^{\alpha}B_{n}(f;x)-{}_{0}^{C}D_{x}^{\alpha}f(x)|
=\mathcal{O}(h),\ h=\frac{1}{n}$.

\begin{theorem}
\label{MThm}
Let $m-1\leq\alpha< m$ and $f\in AC^{m+2}[0, 1]$.
Then, for any $x\in[0, 1]$,
\begin{equation}
\label{order}
\lim_{n\rightarrow\infty}\left[2n(_{0}^CD_{x}^{\alpha}B_{n}(f;x)
-{}_{0}^{C}D_{x}^{\alpha}f(x))\right]
={}_{0}^{C}D_{x}^{\alpha}(x(1-x)f''(x)).
\end{equation}
\end{theorem}

\begin{proof}
The result follows from Theorem~\ref{thm11}, the properties
of the left sided Caputo fractional derivative,
and the fact that under assumptions on $f$ the right-hand side of \eqref{order}
exists almost everywhere on $[0,1]$ (see, e.g., \cite{KilSriTru:06:The}).
\end{proof}

Similar results to those of Theorem~\ref{MThCa} and Theorem~\ref{MThm}
hold for the right sided Caputo fractional derivatives:

\begin{theorem}
\label{MThCa1}
If $m-1\leq\alpha< m$,  $f\in C^m[0,1]$, and $\varepsilon>0$, then
$$
\|_{x}^CD_{1}^{\alpha}f-{}_{x}^CD_{1}^{\alpha}B_{n}(f)\|<\varepsilon.
$$
\end{theorem}

\begin{theorem}
Let $m-1\leq\alpha< m$ and $f\in AC^{m+2}[0, 1]$. Then, for
any $x\in[0, 1]$,
\begin{equation*}
\lim_{n\rightarrow\infty}\left[2n({}_{x}^CD_{1}^{\alpha}B_{n}(f;x)
-{}_{x}^{C}D_{1}^{\alpha}f(x))\right]
={}_{x}^{C}D_{1}^{\alpha}(x(1-x)f''(x)).
\end{equation*}
\end{theorem}


\subsection{Left and right sided Riemann--Liouville fractional integrals of $\mathbf{B_n(f;x)}$}

In order to derive the left and right sided Riemann--Liouville fractional integrals
of the Bernstein polynomials, we replace $\alpha$ by $-\alpha$ in \eqref{pre1} and \eqref{pre2}:
\[
{}_0I_x
B_{n}(f;x)=\sum_{i=0}^{n}\sum_{j=0}^{n-i}\binom{n}{i}\binom{n-i}{j}f\left(\frac
i
n\right)\frac{(-1)^j\Gamma(i+j+1)}{\Gamma(i+j+1+\alpha)}x^{i+j+\alpha}
\]
and
\[
{}_xI_1
B_{n}(f;x)=\sum_{i=0}^{n}\sum_{j=0}^{i}\binom{n}{i}\binom{i}{j}f\left(\frac
i
n\right)\frac{(-1)^j\Gamma(n-i+j+1)}{\Gamma(n-i+j+1+\alpha)}(1-x)^{n-i+j+\alpha}.
\]

\begin{theorem}
\label{MThLRL}
If $m-1\leq\alpha< m$, $f\in C[0,1]$, and $\varepsilon>0$, then
\[
\|_{0}I_{x}^{\alpha}f-{}_0I_x^{\alpha}B_{n}(f)\|<\varepsilon
\]
uniformly on the interval $[0,1]$.
\end{theorem}

\begin{proof}
Using Theorem~\ref{thm1}, we have $\|f-B_n(f)\|<\epsilon$.
Therefore,
\begin{equation*}
\begin{split}
\bigl| {_{0}I_{x}^{\alpha}}f(x) - {}_0I_x^{\alpha} B_{n}(f;x)\bigr|
&=\left|\frac{1}{\Gamma(\alpha)}\left(\int_0^x(x-t)^{\alpha-1}(f(t)-B_n(f;t))\,dt\right)\right|\\
&\leq\frac{1}{\Gamma(\alpha)}\left(\int_0^x(x-t)^{\alpha-1}\left|f(t)-B_n(f;t)\right|\,dt\right)\\
&\leq\frac{1}{\Gamma(\alpha)}\left(\int_0^x(x-t)^{\alpha-1}\|f-B_n(f)\|\,dt\right)\\
&=\epsilon\ \frac{x^{\alpha}}{\Gamma(\alpha+1)}<\varepsilon.
\end{split}
\end{equation*}
\end{proof}

\begin{theorem}
\label{MThRRL}
If $m-1\leq\alpha< m$, $f\in C[0,1]$, and $\varepsilon>0$, then
\[
\|_{x}I_{1}^{\alpha}f-{}_xI_1^{\alpha}B_{n}(f)\|<\varepsilon
\]
uniformly on the interval $[0,1]$.
\end{theorem}

\begin{proof}
The proof is similar to the proof of Theorem~\ref{MThLRL}.
\end{proof}

\begin{theorem}
\label{intord1}
Let $m-1\leq\alpha< m$ and $f\in AC^{m+2}[0, 1]$. Then, for any
$x\in[0, 1]$,
\begin{equation}
\label{order2}
\lim_{n\rightarrow\infty}\left[2n({}_0I_x^\alpha B_{n}(f;x)
-{}_{0}I_{x}^{\alpha}f(x))\right]
={}_{0}I_{x}^{\alpha}(x(1-x)f''(x)).
\end{equation}
\end{theorem}

\begin{proof}
The expression on the right-hand side of \eqref{order2} exists almost everywhere
on the interval $[0,1]$ \cite{KilSriTru:06:The}. Equality \eqref{order2} follows
from Theorem~\ref{thm11}.
\end{proof}

Analogous result holds for the right sided Riemann--Liouville fractional integral:

\begin{theorem}
\label{intord2}
Let $m-1\leq\alpha< m$ and $f\in AC^{m+2}[0, 1]$. Then, for any
$x\in[0, 1]$,
\begin{equation*}
\lim_{n\rightarrow\infty}\left[2n({}_xI_1^\alpha B_{n}(f;x)
-{}_{x}I_{1}^{\alpha}f(x))\right]
={}_{x}I_{1}^{\alpha}(x(1-x)f''(x)).
\end{equation*}
\end{theorem}


\section{Numerical experiments}
\label{sec:4}

We present two examples in order to illustrate
the numerical usefulness of Theorems~\ref{MThCa}--\ref{intord2}.
To explore the dependence of the error with the parameter $n$,
we use the following definitions:
\begin{equation*}
\begin{split}
E(n,\alpha) &= \max\limits_{1 \leqslant j \leqslant N}
\left |   {}_0^CD_{x}^{\alpha}y(x_j) -{}_0^CD_{x}^{\alpha}B_{n}(y;x_j) \right |,
\quad x_j=jh, \quad j=1,\ldots,N,\\
E(n,-\alpha) &= \max\limits_{1 \leqslant j \leqslant N}
\left |   {}_0I_{x}^{\alpha}y(x_j) -{}_0I_{x}^{\alpha}B_{n}(y;x_j) \right |,
\quad x_j=jh, \quad j=1,\ldots, N.
\end{split}
\end{equation*}
In our simulations we choose $N = 100$.
The \emph{experimentally order of convergence} (EOC) is considered
as in \cite{Kia:The:10}: for ${}_0^CD_{x}^{\alpha}y(x)$
$EOC=\left|\log_{2}{\left|\frac{E(2n,\alpha)}{E(n,\alpha)}\right|}\right|$,
while for ${}_0I_{x}^{\alpha}y(x)$
$EOC=\left|\log_{2}{\left|\frac{E(2n,-\alpha)}{E(n,-\alpha)}\right|}\right|$.
The following notations are used throughout:
$y^{(\alpha)}(x) := {}_0^CD_x^{\alpha}y(x)$,
$y^{(-\alpha)}(x) := {}_0I_x^{\alpha}y(x)$,
$B_{n}^{(-\alpha)}(y;x) := {}_0I_x^{\alpha}B_{n}(y;x)$,
$B_{n}^{(\alpha)}(y;x) := {}_0^CD_x^{\alpha}B_{n}(y;x)$,
$EC(n) := y^{(\alpha)}(x)-B_{n}^{(\alpha)}(y;x)$,
and $EI(n) := y^{(-\alpha)}(x)-B_{n}^{(-\alpha)}(y;x)$.


\begin{example}
\label{EX1}
For our first example we choose $y(x)=e^x$, $x\in[0,1]$.
Using the definition of Caputo fractional derivative
and Riemann--Liouville fractional integral,
\begin{equation*}
\begin{split}
{}_0^CD_x^{\alpha}y(x)&=y^{(\alpha)}(x)=x^{1-\alpha}\sum_{k=0}^{\infty}
\frac{x^k}{\Gamma(k+2-\alpha)},\ \alpha\in(0,1],\\
{}_0I_x^{\alpha}y(x)&=y^{(-\alpha)}(x)=\sum_{k=0}^{\infty}
\frac{x^{k+\alpha}}{\Gamma(k+1+\alpha)},\ \alpha \geq 0.
\end{split}
\end{equation*}
The comparison of $y^{(\alpha)}(x)$ with $B_{n}^{\alpha}(y;x)$ and
$y^{(-\alpha)}(x)$ with $B_{n}^{-\alpha}(y;x)$, for $n=5,10,15,20$ and
$\alpha=\frac{1}{2}$, is shown in Fig.~\ref{Fig.1.} and Fig.~\ref{Fig.1.b},
respectively. In Tables~\ref{EXY.1.1} and~\ref{EXY.1.2}, the values of
$EC(n)$ and $EI(n)$, for $n=40,60,80,100$ and $\alpha=\frac{1}{2}$,
are reported for $x\in[0,1]$. The experimentally order of convergence
for $\alpha=\frac{1}{4}$, $\frac{1}{2}$, and $\frac{3}{4}$ is illustrated
in Tables~\ref{E.r1} and~\ref{E.r11} with $n = 20, 40, 80, 160, 320$.
\begin{figure}
\center\includegraphics[scale=0.35]{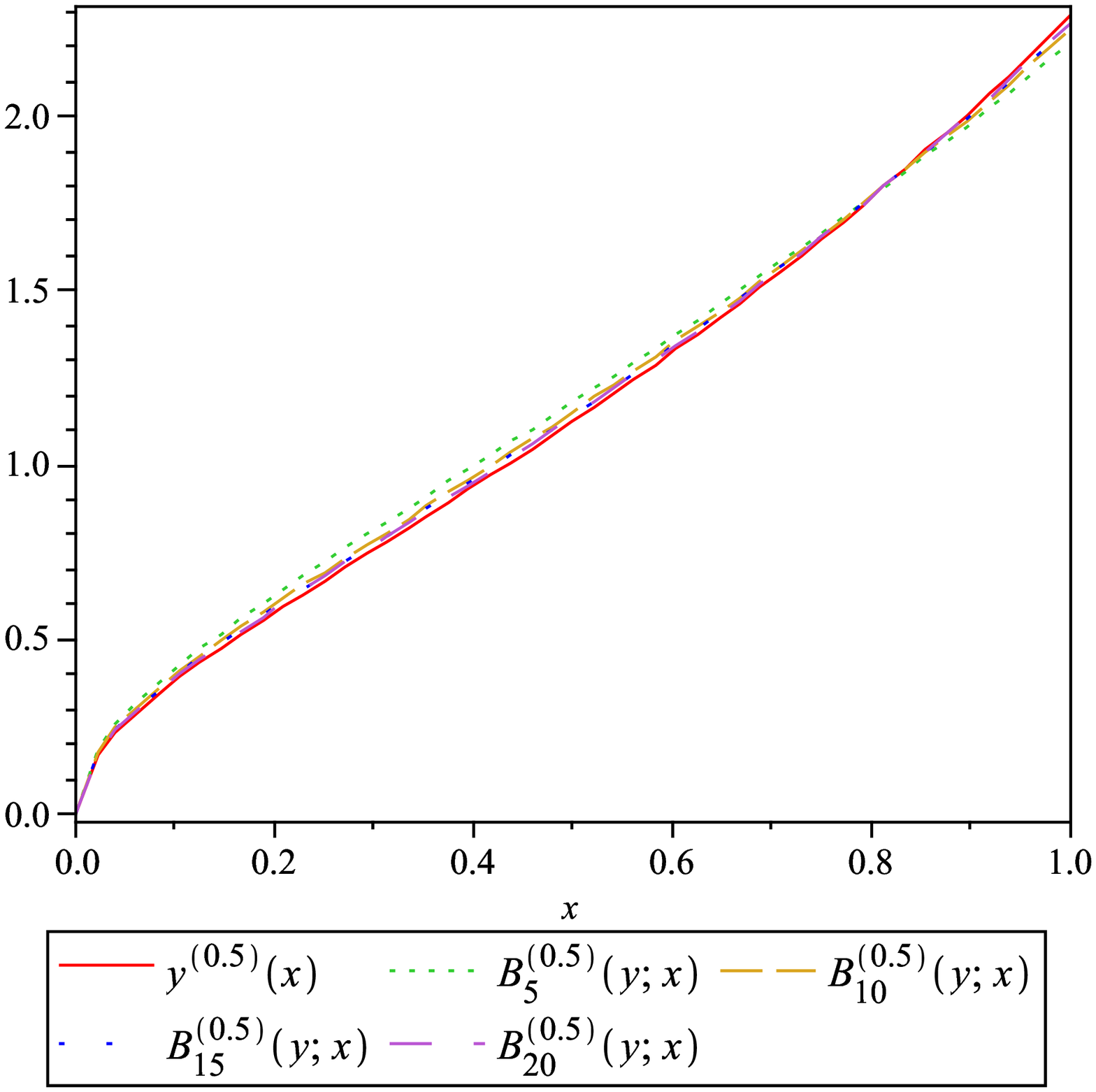}\includegraphics[scale=0.35]{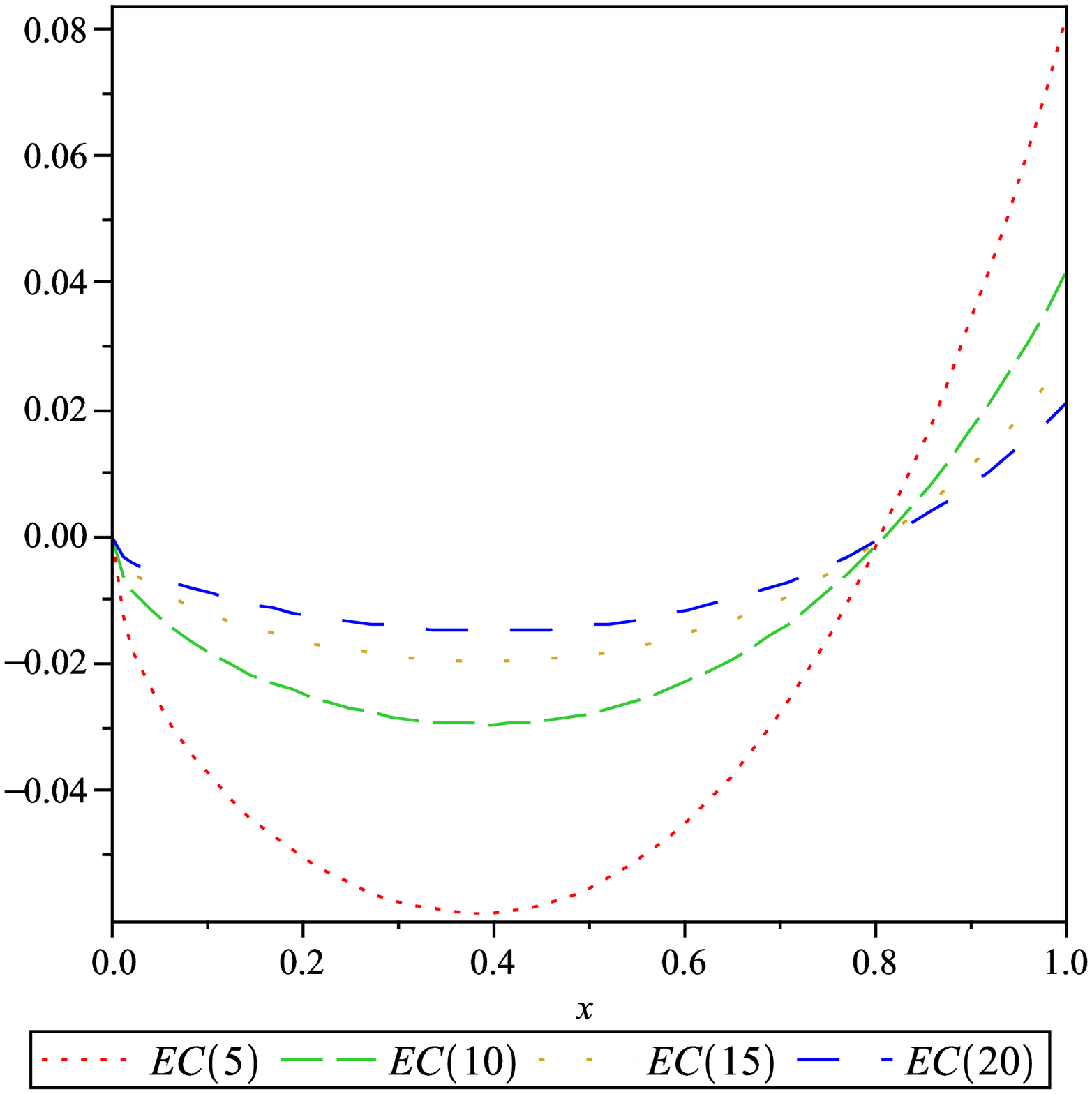}
\caption{Example~\ref{EX1}: $y^{(\alpha)}(x)$ versus $B_{n}^{(\alpha)}(y;x)$ (left)
and plot of $EC(n)$ (right) with $\alpha=\frac{1}{2}$
and $n=5,10,15,20$.} \label{Fig.1.}
\end{figure}
\begin{figure}
\center\includegraphics[scale=0.30]{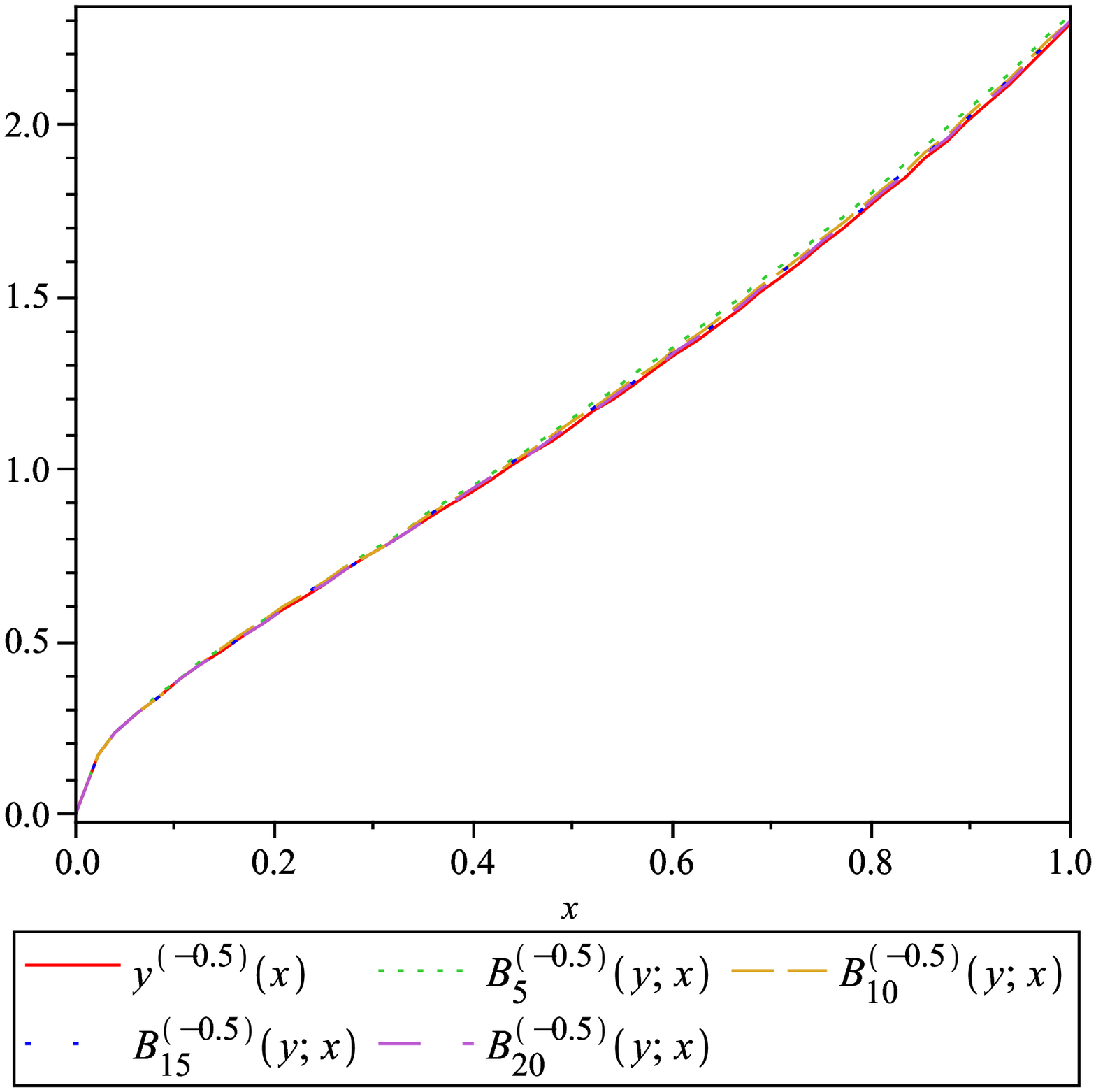}\includegraphics[scale=0.30]{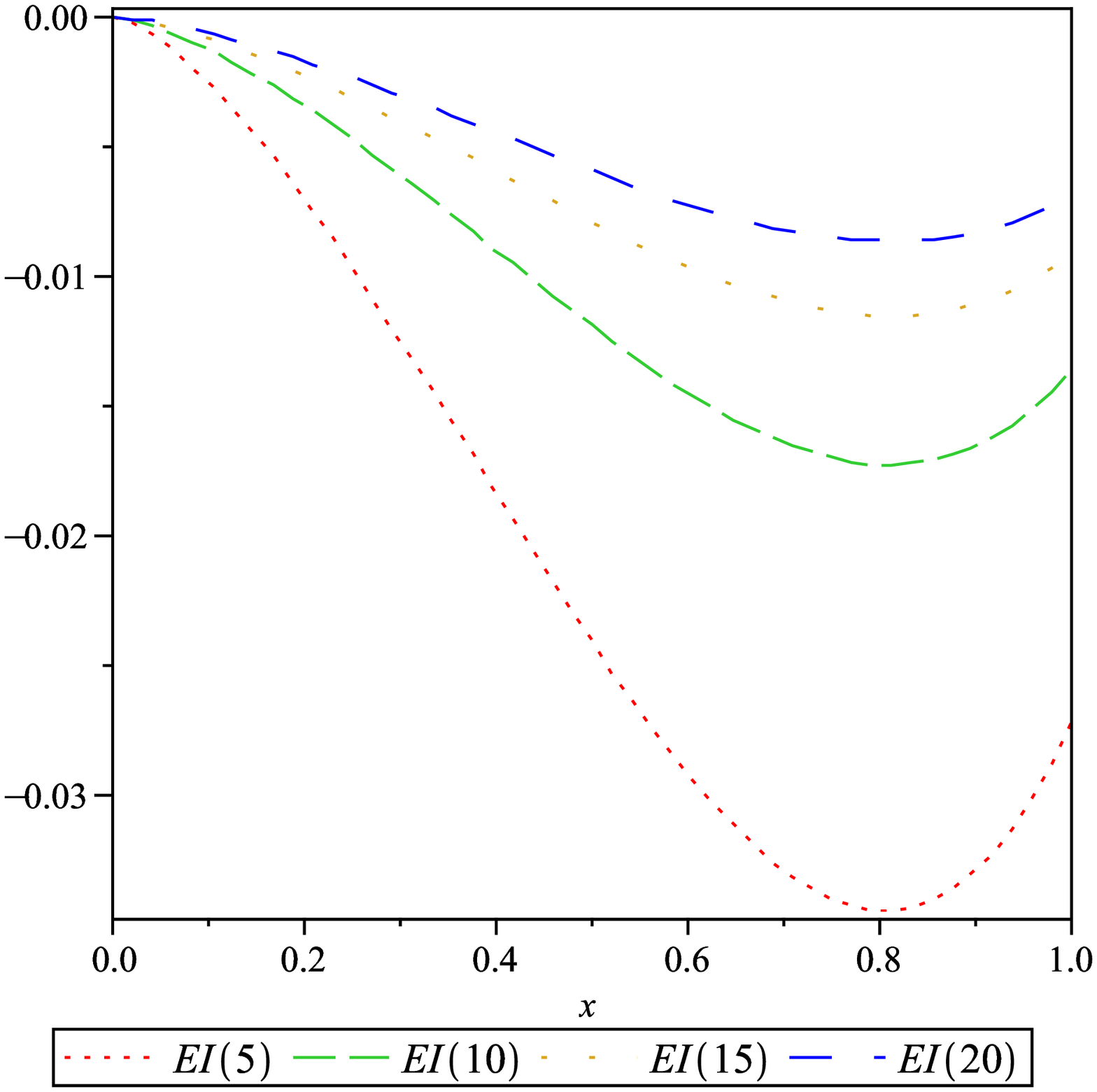}
\caption{Example \ref{EX1}:
$y^{(-\alpha)}(x)$ versus $B_{n}^{(-\alpha)}(y;x)$ (left)
and plot of $EI(n)$ (right) with $\alpha=\frac{1}{2}$
and $n=5,10,15,20$.}\label{Fig.1.b}
\end{figure}
\begin{table}
\centering
\begin{tabular}{ccccc}
\hline
$x$  &  $EC(40)$     & $EC(60)$      & $EC(80)$      & $EC(100)$ \\ \hline
0.0  & 0.0000000000  & 0.0000000000  & 0.0000000000  & 0.0000000000 \\
0.2  & 0.0061036942  & 0.0040621962  & 0.0030440502  & 0.0024339892 \\
0.4  & 0.0073719719  & 0.0049132237  & 0.0036843732  & 0.0029472355 \\
0.6  & 0.0058767960  & 0.0039249810  & 0.0029463970  & 0.0023583970 \\
0.8  & 0.0004944290  & 0.0003392010  & 0.0002580000  & 0.0002081350 \\
1.0  & 0.0106309420  & 0.0070992270  & 0.0053288670  & 0.0042652620 \\ \hline
\end{tabular}
\caption{Example~\ref{EX1}: values of $EC(n)$ for $n=40,60,80,100$,
$\alpha=\frac{1}{2}$, and some values of $x\in[0,1]$.}\label{EXY.1.1}
\end{table}
\begin{table}
\centering
\begin{tabular}{ccccc}
\hline
$x$  &  $EI(40)$     & $EI(60)$      & $EI(80)$      & $EI(100)$ \\ \hline
0.0  & 0.0000000000  & 0.0000000000  & 0.0000000000  & 0.0000000000 \\
0.2  & 0.0008334426  & 0.0005544267  & 0.0004153698  & 0.0003320800 \\
0.4  & 0.0022223574  & 0.0014794349  & 0.0011087740  & 0.0008866339 \\
0.6  & 0.0036002970  & 0.0023986370  & 0.0017983900  & 0.0014384260 \\
0.8  & 0.0043154540  & 0.0028772970  & 0.0021580890  & 0.0017265270 \\
1.0  & 0.0034169630  & 0.0022786810  & 0.0017092750  & 0.0013675460\\ \hline
\end{tabular}
\caption{Example~\ref{EX1}: values of $EI(n)$ for $n=40,60,80,100$,
$\alpha=\frac{1}{2}$, and some values of $x\in[0,1]$.}
\label{EXY.1.2}
\end{table}
\begin{table}
\centering
\begin{tabular}{ccccccc}
\hline
$ n\ $&  &$EOC$, $\alpha=\frac{1}{4}$ & & $EOC$, $\alpha=\frac{1}{2}$ &  &$EOC$, $\alpha=\frac{3}{4}$\\ \hline
${20}$& \qquad\qquad & 0.9945716137 &\qquad & 0.9927445216 & \qquad &0.9905964160  \\
${40}$&  & 0.9972834530 &  & 0.9963632907 &  &0.9952824963  \\
${80}$ &  & 0.9986406642 &  &0.9981985963  &  &0.9976418495\\
${160}$&  & 0.9993064563&  &0.9990896057  &  &0.9987894922 \\
${320}$&  & 0.9996330493 & & 0.9994812690 &  & 0.9990779310\\ \hline
\end{tabular}
\caption{Example~\ref{EX1}: experimentally determined order of convergence for
${}_0^CD_{x}^{\alpha}e^x$, $\alpha=\frac{1}{4}$, $\frac{1}{2}$, $\frac{3}{4}$,
and different values of $n$.}\label{E.r1}
\end{table}
\begin{table}
\centering
\begin{tabular}{ccccccc}
\hline
$ n\ $&  & $EOC$, $\alpha=\frac{1}{4}$ & & $EOC$,
$\alpha=\frac{1}{2}$ &  & $EOC$, $\alpha=\frac{3}{4}$  \\ \hline
 ${20}$& \qquad\qquad & 0.9974779044 &\qquad & 0.9986421496 & \qquad &0.9996580722  \\
 ${40}$&  & 0.9987433922 &  & 0.9993300993 &  &0.9998404485  \\
 ${80}$ &  & 0.9993742741 &  &0.9996624232  &  &0.9999198418\\
 ${160}$&  & 0.9996853779&  &0.9998531789  &  &0.9999704282 \\
 ${320}$&  & 0.9998513835 & & 0.9998852549 &  & 0.9999719852\\ \hline
\end{tabular}
\caption{Example~\ref{EX1}: experimentally determined order of convergence for
${}_0I_{x}^{\alpha}e^x$, $\alpha=\frac{1}{4}$, $\frac{1}{2}$, $\frac{3}{4}$,
and different values of $n$.}\label{E.r11}
\end{table}
\end{example}


\begin{example}
\label{EX2}
Let $y(x)=\sin(x)$, $x\in[0,1]$.
The Caputo fractional derivative
and the Riemann--Liouville fractional integral
of $y$ are given in \cite{TakHir:09:Qua}:
\begin{gather*}
{}_0^CD_x^{\alpha}y(x)=y^{(\alpha)}(x)=x^{1-\alpha}\sum_{k=0}^{\infty}
\frac{(-1)^kx^{2k}}{\Gamma(2k+2-\alpha)},\ \alpha\in(0,1],\\
{}_0I_x^{\alpha}y(x)=y^{(-\alpha)}(x)=\sum_{k=0}^{\infty}
\frac{(-1)^kx^{2k+1+\alpha}}{\Gamma(2k+2+\alpha)},\ \alpha\geq 0.
\end{gather*}
The comparison of $y^{(\alpha)}(x)$ with $B_{n}^{(\alpha)}(y;x)$ and
$y^{(-\alpha)}(x)$ with $B_{n}^{(-\alpha)}(y;x)$, for $n=5,10,15,20$
and $\alpha=\frac{3}{4}$, is shown in Fig.~\ref{Fig.3} and Fig.~\ref{Fig.4},
respectively. In Tables~\ref{EXY.2.1} and~\ref{EXY.2.2},
$EC(n)$ and $EI(n)$ for $n=40,60,80,100$,
$\alpha=\frac{3}{4}$, and some values of $x\in[0,1]$,
are reported. The experimentally order of convergence for
$\alpha=\frac{1}{4}$, $\frac{1}{2}$, $\frac{3}{4}$ is shown
in Table~\ref{E.r2} for ${}_0^CD_{x}^{\alpha}\sin(x)$
and in Table~\ref{E.r3} for ${}_0I_{x}^{\alpha}\sin(x)$.
\begin{figure}
\center
\includegraphics[scale=0.30]{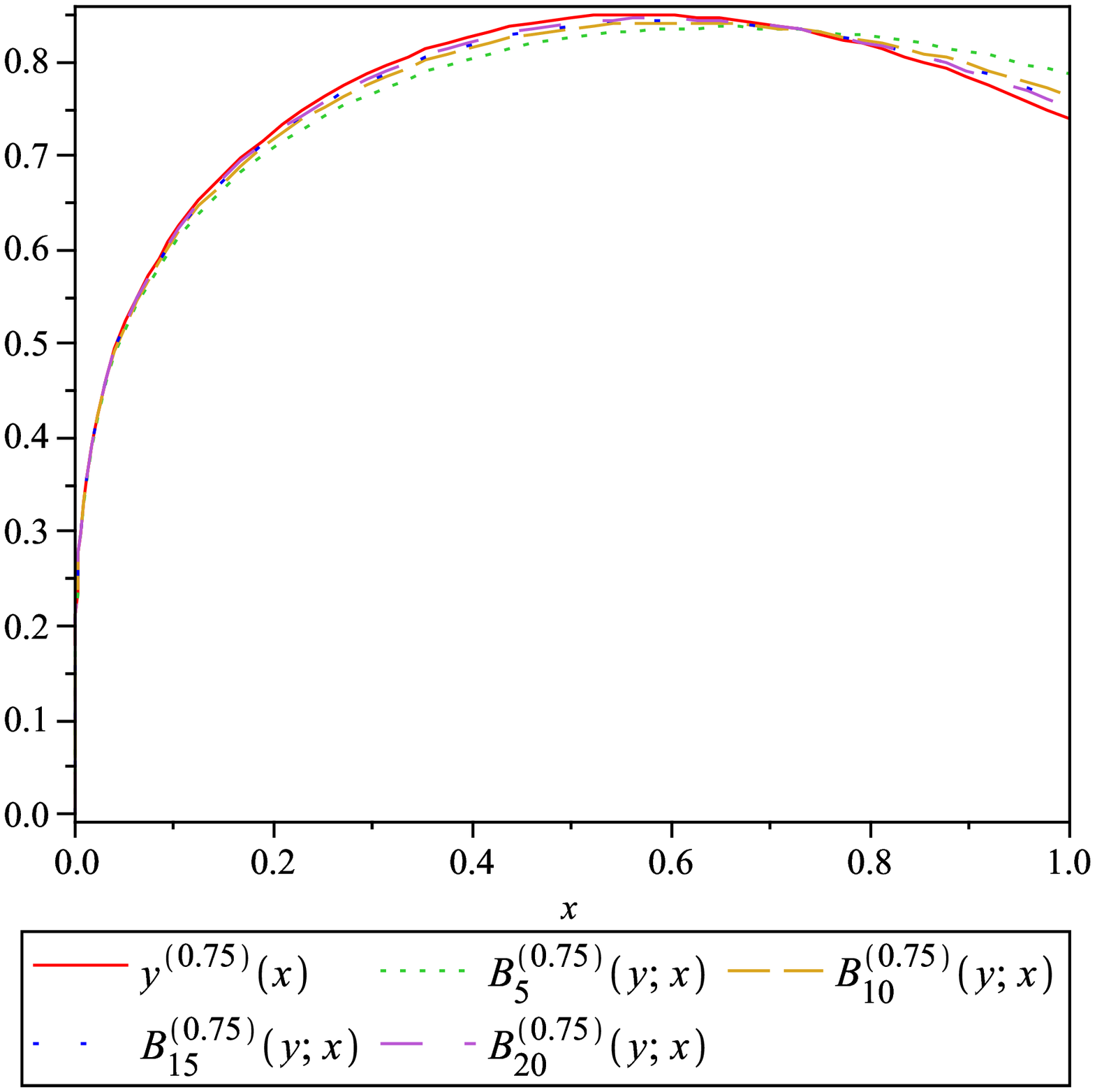}\includegraphics[scale=0.30]{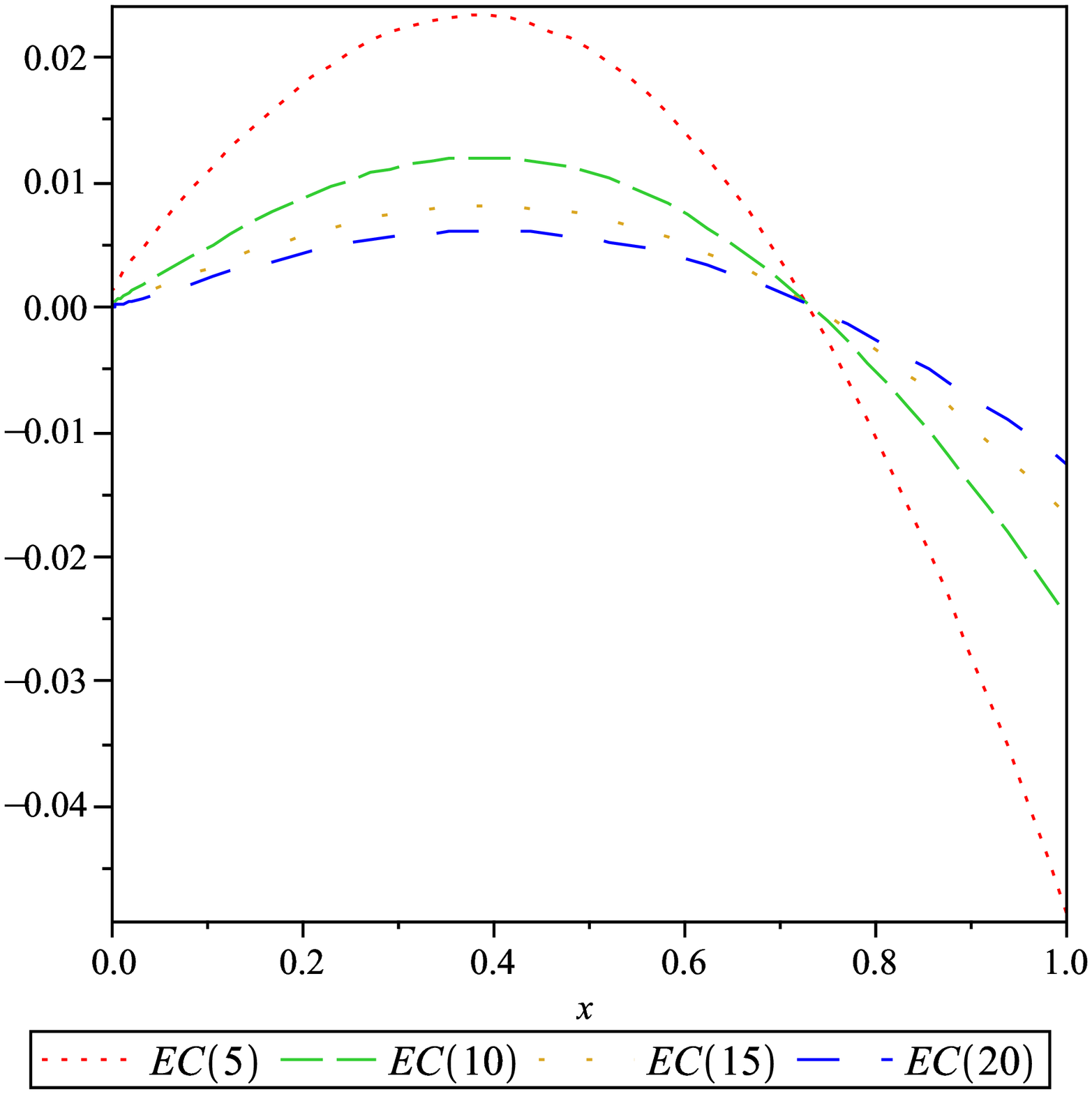}
\caption{Example~\ref{EX2}: $y^{(\alpha)}(x)$ versus
$B_{n}^{(\alpha)}(y;x)$ (left) and $EC(n)$ (right) for
$\alpha=\frac{3}{4}$ and $n=5,10,15,20$.}\label{Fig.3}
\end{figure}
\begin{figure}
\center
\includegraphics[scale=0.30]{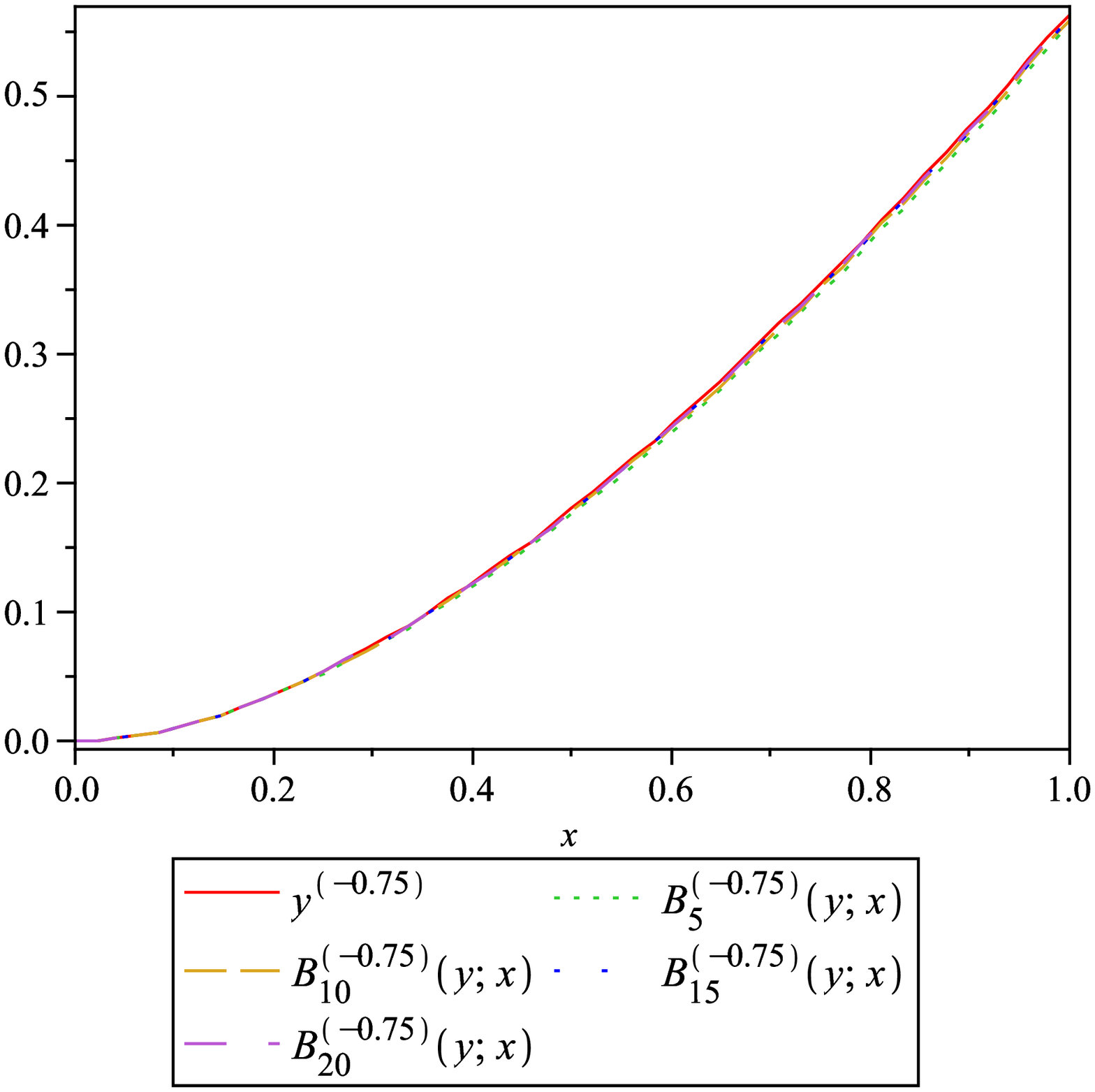}\includegraphics[scale=0.30]{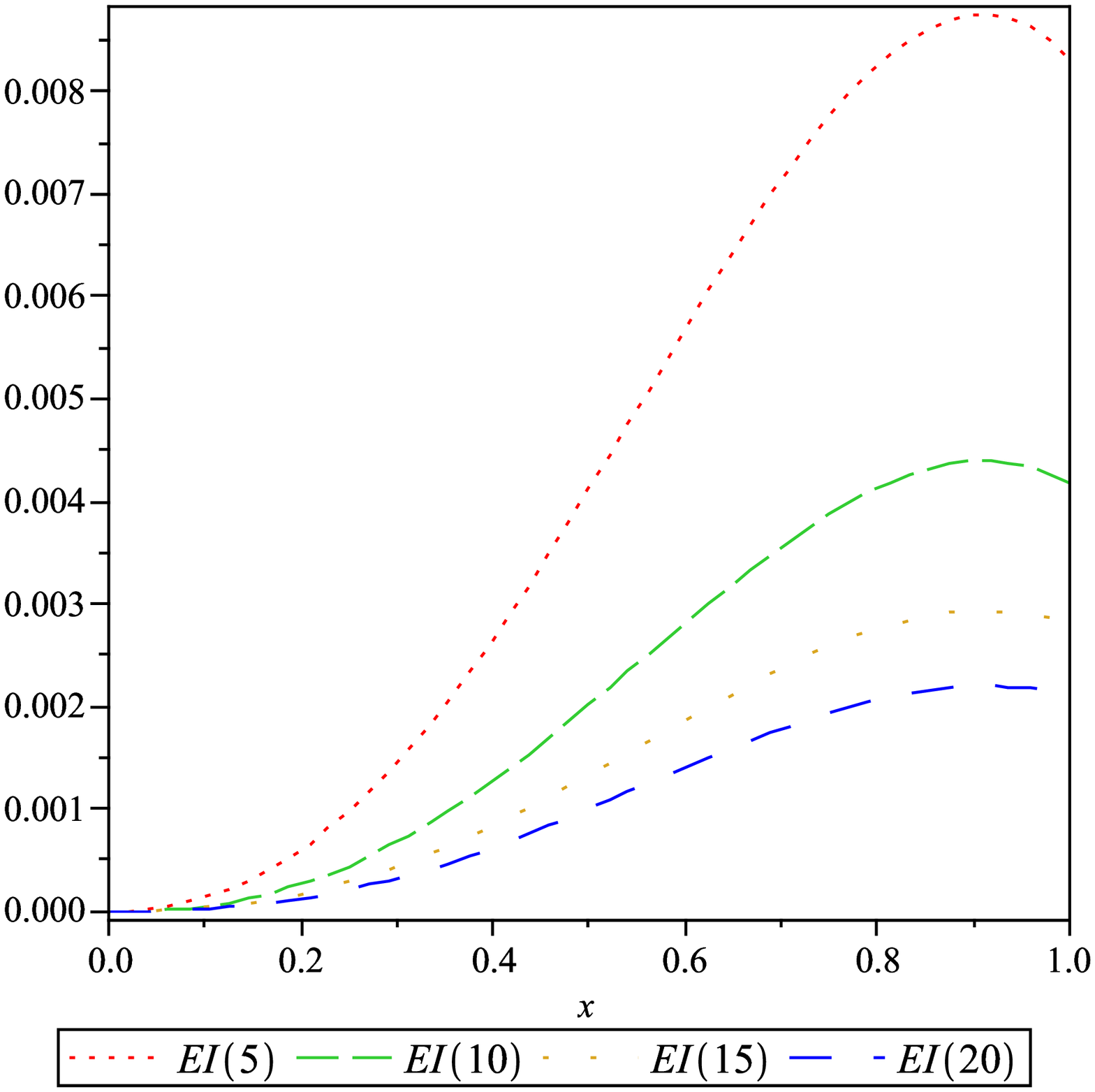}
\caption{Example~\ref{EX2}: $y^{(-\alpha)}(x)$ versus
$B_{n}^{(-\alpha)}(y;x)$ (left) and  $EI(n)$ (right) for
$\alpha=\frac{3}{4}$ and $n=5,10,15,20$.}\label{Fig.4}
\end{figure}
\begin{table}
\centering
\begin{tabular}{ccccc}\hline
$x$  &  $EC(40)$     & $EC(60)$      & $EC(80)$      & $EC(100)$    \\ \hline
0.0  & 0.0000000000  & 0.0000000000  & 0.0000000000  & 0.0000000000 \\
0.2  & 0.0021517476  & 0.0014317995  & 0.0010728366  & 0.0008577819 \\
0.4  & 0.0030879772  & 0.0020645865  & 0.0015506726  & 0.0012416113 \\
0.6  & 0.0019715930  & 0.0013211308  & 0.0009933784  & 0.0007959183 \\
0.8  & 0.0012727043  & 0.0008483037  & 0.0006361733  & 0.0005089138 \\
1.0  & 0.0063047021  & 0.0042104896  & 0.0031606168  & 0.0025297984 \\ \hline
\end{tabular}
\caption{Example~\ref{EX2}: comparison of $EC(n)$ for $n=40$, 60, 80, 100,
$\alpha=0.75$, and some values of $x\in[0,1]$.}\label{EXY.2.1}
\end{table}
\begin{table}
\centering
\begin{tabular}{ccccc}\hline
$x$  &  $EI(40)$     & $EI(60)$      & $EI(80)$      & $EI(100)$    \\ \hline
0.0  & 0.0000000000  & 0.0000000000  & 0.0000000000  & 0.0000000000 \\
0.2  & 0.0000588363  & 0.0000387162  & 0.0000288465  & 0.0000229857 \\
0.4  & 0.0003075654  & 0.0002042553  & 0.0001528957  & 0.0001221744 \\
0.6  & 0.0006967566  & 0.0004640866  & 0.0003479090  & 0.0002782522 \\
0.8  & 0.0010354049  & 0.0006904372  & 0.0005178911  & 0.0004143439 \\
1.0  & 0.0010486693  & 0.0006994101  & 0.0005246694  & 0.0004197893 \\ \hline
\end{tabular}
\caption{Example~\ref{EX2}: comparison of $EI(n)$ for $n=40$, 60, 80, 100,
$\alpha=0.75$, and some values of $x\in[0,1]$.}\label{EXY.2.2}
\end{table}
\begin{table}
\centering
\begin{tabular}{ccccccc} \hline
$ n\ $&  & $EOC$, $\alpha=\frac{1}{4}$ & & $EOC$,
$\alpha=\frac{1}{2}$ &  & $EOC$, $\alpha=\frac{3}{4}$  \\ \hline
 ${20}$&\qquad\qquad  & 0.9937752559 &\qquad & 0.9929561676 &\qquad  &0.9923368788  \\
 ${40}$&  & 0.9969076505 &  & 0.9965126525 &  &0.9962296261  \\
 ${80}$ &  & 0.9984573422 &  &0.9982238029  &  &0.9981122106\\
 ${160}$&  & 0.9992402809 &  &0.9991070234  &  &0.9990444572 \\
 ${320}$&  & 0.9994769560 & & 0.9996091945 &  & 0.9998994139\\ \hline
\end{tabular}
\caption{Example~\ref{EX2}: experimentally determined order of convergence
for ${}_0^CD_{x}^{\alpha}\sin(x)$, $\alpha=\frac{1}{4}$,
$\frac{1}{2}$, $\frac{3}{4}$, and different $n$.}\label{E.r2}
\end{table}
\begin{table}
\centering
\begin{tabular}{ccccccc}\hline
$n$ &  & $EOC$, $\alpha=\frac{1}{4}$ & & $EOC$, $\alpha=\frac{1}{2}$ &  &$EOC$, $\alpha=\frac{3}{4}$\\ \hline
 ${20}$&\qquad\qquad  & 0.9958096867 &\qquad & 0.9969600807 &\qquad  &0.9981760620  \\
 ${40}$&  & 0.9979038936 &  & 0.9984739450 &  &0.9990790991  \\
 ${80}$ &  & 0.9989581171 &  &0.9992369945  &  &0.9995370214\\
 ${160}$&  & 0.9994912812 &  &0.9996141316  &  &0.9997773614 \\
 ${320}$&  & 0.9997385874 & & 0.9998309959 &  & 0.9999065568\\ \hline
\end{tabular}
\caption{Example~\ref{EX2}: experimentally determined order
of convergence for ${}_0I_{x}^{\alpha}\sin(x)$, $\alpha=\frac{1}{4}$,
$\frac{1}{2}$, $\frac{3}{4}$, and different $n$.}\label{E.r3}
\end{table}
\end{example}


\section{Application to fractional differential equations}
\label{sec:5}

Consider the nonlinear fractional differential equation
\begin{equation}
\label{prob}
{}_0^CD_{t}^{\alpha}x(t)=f(t,x(t))
\end{equation}
subject to $x^{(k)}(0)=0$, $k=0,\ldots,m-1$,
where $m-1\leq \alpha<m$ and $t\in[0,1]$.
Note that a fractional differential equation \eqref{prob}
with nonzero initial conditions $x^{(k)}(0)=x_k$
can be easily transformed to a problem \eqref{prob}
with vanishing initial conditions. The advantage of considering zero initial values
of $x$ and its derivatives up to order $m-1$ is that under such conditions
the Riemann--Liouville and the Caputo derivatives coincide.
Let $x(t)\approx x_n(t)=\sum_{i=m}^{n}c_i\ t^{i}(1-t)^{n-i}$,
where $c_i=x(i/n)$, $i=m,\ldots,n$, are unknown coefficients to be determined.
Using \eqref{pre1},
\begin{equation*}
{}^{C}_{0}D_{t}^{\alpha}x(t)\approx {} ^{C}_{0}D_{t}^{\alpha}x_n(t)
=\sum_{i=m}^{n}\sum_{j=0}^{n-i}c_i\
\binom{n-i}{j}\frac{(-1)^j\Gamma(i+j+1)}{\Gamma(i+j+1-\alpha)}t^{i+j-\alpha}.
\end{equation*}
We approximate \eqref{prob} by ${} ^{C}_{0}D_{t}^{\alpha}x_n(t)=f(t,x_n(t))$
and, considering equidistant nodes $t_j =jh$ in the interval $[0, 1]$,
$j =0, \ldots n-m$, $h=\frac{1}{n}$, we transform this system into
${} ^{C}_{0}D_{t}^{\alpha}x_n(t_j)=f(t_j,x_n(t_j))$, $j=0,\ldots,n-m$.
The $n-m+1$ unknown coefficients $c_i$, $i=m, \ldots n$, are found
by solving this algebraic system of $n-m+1$ equations.


\subsection{Convergence and stability}

We start to prove an important result about the error committed
when solving the fractional differential equation
\eqref{prob} with our approximate method (Theorem~\ref{Conve}).
The second result asserts that the approximate solutions of \eqref{prob}
are stable with respect to the right-hand side of the fractional
differential equation (Theorem~\ref{thm:stb}). The proof
of both results make use of the following Gronwall-type result for
fractional integral equations:

\begin{lemma}[\cite{Kia:The:10}]
\label{Gron}
Let $\alpha$, $T$, $\epsilon_1$, $\epsilon_2\in\Bbb{R}^{+}$.
Moreover, assume that $\delta:[0,T]\rightarrow\Bbb{R}$ is a
continuous function satisfying the inequality
\begin{equation*}
|\delta(t)|\leq\epsilon_1+\frac{\epsilon_2}{\Gamma(\alpha)}
\int_0^t(t-x)^{\alpha-1}|\delta(x)|\,dx
\end{equation*}
for all $t\in[0,T]$. Then,
$|\delta(t)|\leq\epsilon_1E_{\alpha,1}(\epsilon_2t^\alpha)$
for $t\in[0,T]$.
\end{lemma}

\begin{theorem}
\label{Conve}
Let $x$ be the solution of
\begin{equation}
\label{Pro}
{}_0^CD_{t}^{\alpha}x(t)=f(t,x(t)),\ x^{(k)}(0)=0,
\quad k=0,\ldots, \lceil\alpha\rceil-1,
\end{equation}
where $f$ satisfies a Lipschitz condition in its second argument
on $[0,1]$. If $x_n(t)=B_n(x;t)$ is the approximate solution of
\eqref{Pro}, then $x_n(t)$ is convergent to $x(t)$ as
$n\rightarrow\infty$, and $|x_n(t)-x(t)|\leq \mathcal{O}(h)$.
\end{theorem}

\begin{proof}
As in the proof of  Theorem~\ref{MThm}, if
$x^{(k)}_n(0)=0$, $k=0,\ldots, \lceil\alpha\rceil-1$, then
$$
{}_0^CD_{t}^{\alpha}x_n(t)={}_0D_{t}^{\alpha}x_n(t)=f(t,x_n(t))+\mathcal{O}(h).
$$
If we write $x$ and $x_n$ in the form of the
equivalent Volterra integral equation
\begin{equation*}
x(t)=\frac{1}{\Gamma(\alpha)}\int_0^t(t-z)^{\alpha-1}f(z,x(z))\,dz
\end{equation*}
and
\begin{equation*}
x_n(t)=\frac{1}{\Gamma(\alpha)}\int_0^t(t-z)^{\alpha-1}\left(f(z,x_n(z))
+\mathcal{O}(h)\right)\,dz,
\end{equation*}
subtracting we obtain the relation
\begin{equation*}
x(t)-x_n(t) =\frac{1}{\Gamma(\alpha)}\int_0^t(t-z)^{\alpha-1}\left(f(z,x(z))-f(z,x_n(z))\right)\,dz
+\mathcal{O}(h)\left(\frac{1}{\Gamma(\alpha)}\int_0^t(t-z)^{\alpha-1}\,dz\right).
\end{equation*}
Using the Lipschitz condition on $f$ in the first term on
the right-hand side,
\begin{equation*}
\left|\frac{1}{\Gamma(\alpha)}\int_0^t(t-z)^{\alpha-1}\left(f(z,x(z))-f(z,x_n(z))\right)\,dz\right|
< \Lambda\left(\frac{1}{\Gamma(\alpha)}\int_0^t(t-z)^{\alpha-1}| x(z)-x_n(z)|\,dz\right) .
\end{equation*}
If we put $Y(t)=|x(t)-x_n(t)|$, then we have
\begin{equation}
\label{1}
|Y(t)|\leq \mathcal{O}(h)+\Lambda\left(\frac{1}{\Gamma(\alpha)}
\int_0^t(t-z)^{\alpha-1}Y(z)\,dz\right).
\end{equation}
In view of Lemma~\ref{Gron}, \eqref{1} allow us to conclude that
$|Y(t)|\leq\mathcal{O}(h)E_{1,\alpha}(\Lambda t^\alpha)$
for $t\in[0,1]$, and the proof is complete.
\end{proof}

\begin{theorem}
\label{thm:stb}
Let $x_n(t)$ and $x'_n(t)$ be approximate solutions of \eqref{Pro}
with the right-hand side of the fractional differential equation
given by $f$ and $f'$, respectively. If $f$ and $f'$ satisfy
a Lipschitz condition in its second argument on $[0,1]$, then
$|x_n(t)-x'_n(t)|\leq C\|f-f'\|$ for any $n$.
\end{theorem}

\begin{proof}
As in the proof of  Theorem~\ref{MThm}, if
$x^{(k)}_n(0)=0,\ k=0,1,\ldots,
\lceil\alpha\rceil-1$, then
${}_0D_{t}^{\alpha}x_n(t)=f(t, x_n(t))+\mathcal{O}(h)$
and ${}_0D_{t}^{\alpha}x'_n(t)=f'(t,x'_n(t))+\mathcal{O}(h)$.
Now, we write the approximate solutions $x_n$ and $x'_n$ in
the Volterra integral equation form:
\begin{equation}
\label{V11}
x_n(t)=\frac{1}{\Gamma(\alpha)}\int_0^t(t-z)^{\alpha-1}\left(f(z,x_n(z))
+\mathcal{O}(h)\right)\,dz
\end{equation}
and
\begin{equation}
\label{V22}
x'_n(t)=\frac{1}{\Gamma(\alpha)}\int_0^t(t-z)^{\alpha-1}\left(f'(z,x'_n(z))
+\mathcal{O}(h)\right)\,dz.
\end{equation}
Subtracting \eqref{V22} from \eqref{V11}, we obtain the relation
\begin{equation*}
\begin{split}
x_n(t)-x'_n(t)
&= \int_0^t\frac{f(z,x_n(z))-f'(z,x'_n(z))}{\Gamma(\alpha){(t-z)^{1-\alpha}}}\,dz
+\mathcal{O}(h)\left(\frac{1}{\Gamma(\alpha)}\int_0^t(t-z)^{\alpha-1}\,dz\right)\\
&=\frac{1}{\Gamma(\alpha)}\int_0^t\frac{f(z,x_n(z))-f(z,x'_n(z))}{{(t-z)^{1-\alpha}}}\,dz\\
&\quad +\frac{1}{\Gamma(\alpha)}\int_0^t\frac{f(z,x'_n(z))-f'(z,x'_n(z))}{{(t-z)^{1-\alpha}}}\,dz
+\mathcal{O}(h)\frac{t^\alpha}{\Gamma(1+\alpha)}, \quad t\in[0,1].
\end{split}
\end{equation*}
Using the Lipschitz conditions on $f$ in the first term on the right-hand side, we get
\begin{equation*}
\left|\frac{1}{\Gamma(\alpha)}\int_0^t\frac{f(z,x_n(z))-f(z,x'_n(z))}{{(t-z)^{1-\alpha}}}\,dz\right|
< \frac{\Lambda}{\Gamma(\alpha)}\int_0^t(t-z)^{\alpha-1}| x_n(z)-x'_n(z)|\,dz
\end{equation*}
by evaluating the integral representations of the fractional
derivatives of $x_n(t)$ and $x'_n(t)$. For the second term on
the right-hand side, we have
\begin{equation*}
\begin{split}
\left|\frac{1}{\Gamma(\alpha)}\int_0^t
\frac{f(z,x'_n(z))-f'(z,x'_n(z))}{{(t-z)^{1-\alpha}}}\,dz\right|
&<L\|f-f'\|\left(\frac{1}{\Gamma(\alpha)}\int_0^t(t-z)^{\alpha-1}\,dz\right)\\
&=L\frac{t^\alpha}{\Gamma(1+\alpha)}\|f-f'\|
\end{split}
\end{equation*}
for $t\in[0,1]$. If we put $Y(t)=|x_n(t)-x'_n(t)|$, then
\begin{equation}
\label{111}
|Y(t)| \leq \mathcal{O}(h)+L'\|f-f'\|+\Lambda\left(\frac{1}{\Gamma(\alpha)}
\int_0^t(t-z)^{\alpha-1}Y(z)\,dz\right).
\end{equation}
In view of Lemma~\ref{Gron}, \eqref{111} allow us to conclude that
\begin{equation*}
|Y(t)|\leq(\mathcal{O}(h)+L'\|f-f'\|)E_{1,\alpha}(\Lambda
t^\alpha)\leq C\|f-f'\|, \quad t\in[0,1].
\end{equation*}
\end{proof}


\subsection{Illustrative examples}

We give four examples of application of our method.

\begin{example}
\label{FDEex1}
The following problem is studied in \cite{SchGau:06:Ona}:
$c ({}_0^CD_{t}^{\alpha}x)(t)+kx(t)=f(t)$, $0\leq t\leq1$,
subject to $x(0)=0$ and where $\alpha=\frac{1}{2}$, $c=100$, $k=10$, and $f(t) \equiv 1$.
The exact solution is shown to be
$x_{exact}(t)=\frac{1}{k}\left(1-E_{\alpha,1}\left(-\frac{k}{c}t^\alpha\right)\right)$
\cite{SchGau:06:Ona}. Fig.~\ref{FigFD1} compares the exact solution
with the numerical approximations obtained by our method.
\begin{figure}
\center\includegraphics[scale=0.30]{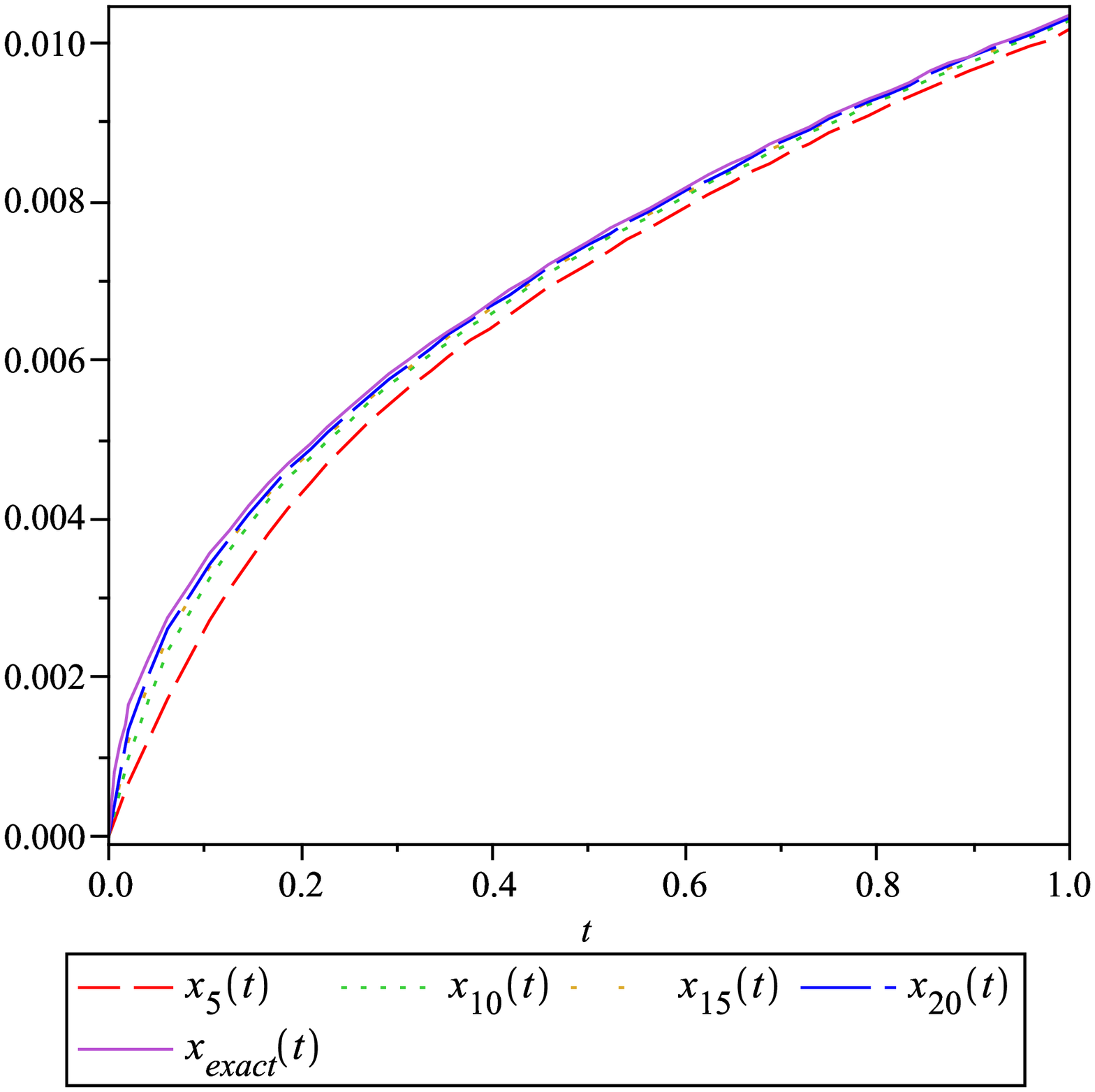}\includegraphics[scale=0.30]{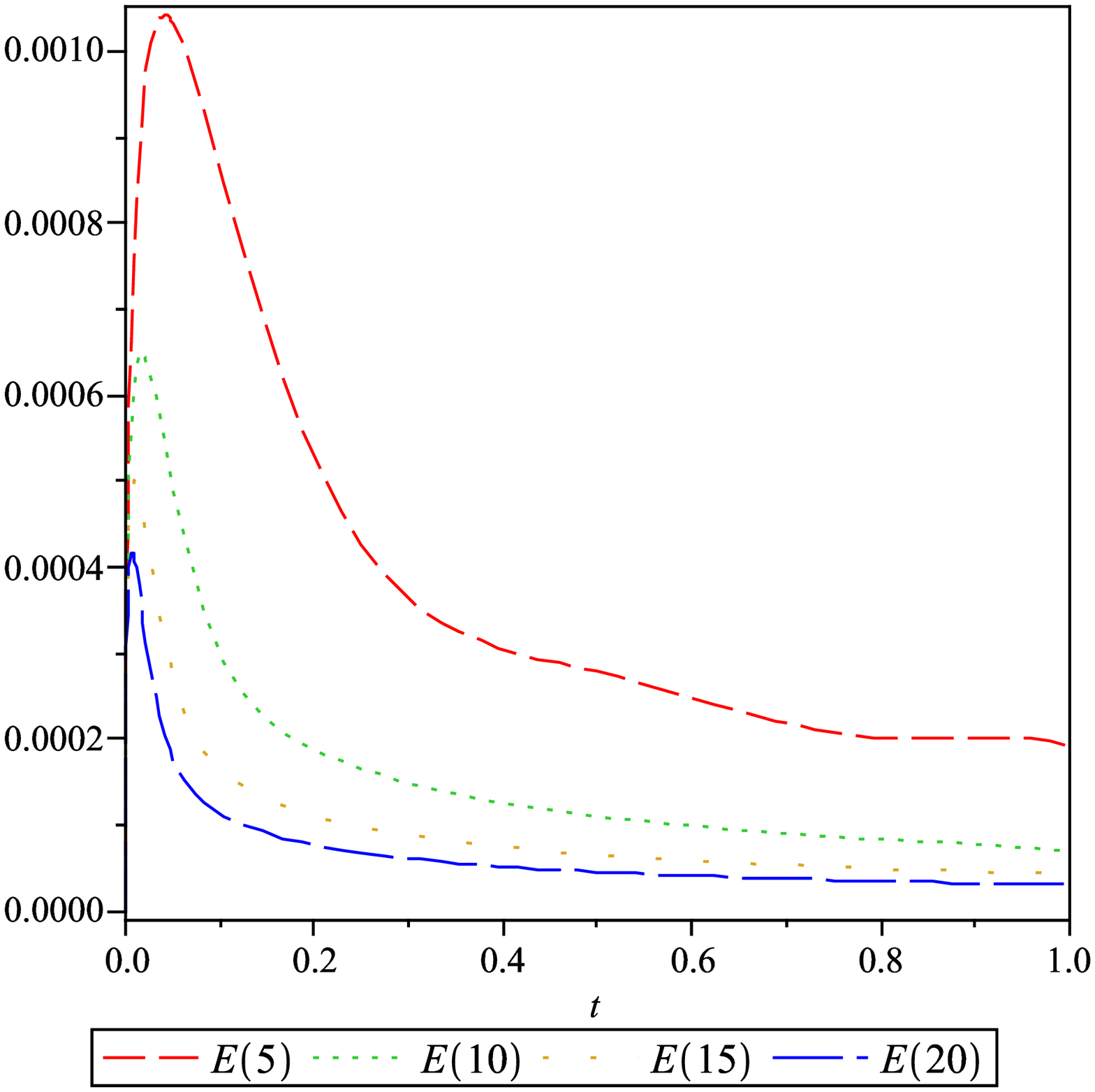}
\caption{Example~\ref{FDEex1}: $x_{exact}$ versus $x_n(t)=B_n(x;t)$ (left)
and $E(n)=|x_{exact}(t)-x_n(t)|$ (right),
$n = 5,10,15,20$.}\label{FigFD1}
\end{figure}
\end{example}

\begin{example}
\label{FDEex2}
As a second example, consider the nonlinear ordinary
differential equation
$$
({}_0^CD_{t}^{\alpha}x)(t)+k(x(t))^2=f(t)
$$
subject to $x(0)=0$ and $x^{(1)}(0)=0$,
where
$$
f(t)= \frac{120 {t}^{5-\alpha}}{\Gamma\left( 6-\alpha\right)}
-\frac{72{t}^{4-\alpha}}{\Gamma\left(5-\alpha\right)}
+\frac{12{t}^{3-\alpha}}{\Gamma\left(4-\alpha\right)}
+k \left({t}^{5}-3\,{t}^{4}+2\,{t}^{3}\right)^{2}.
$$
Following \cite{LinLiu:07:Fra}, we take
$k =1$ and $\alpha = \frac{3}{2}$.
The exact solution is given by
$x_{exact}(t)={t}^{5}-3t^4+2t^3$ \cite{LinLiu:07:Fra}.
Fig.~\ref{FigFD2} plots the results obtained.
\begin{figure}
\center\includegraphics[scale=0.30]{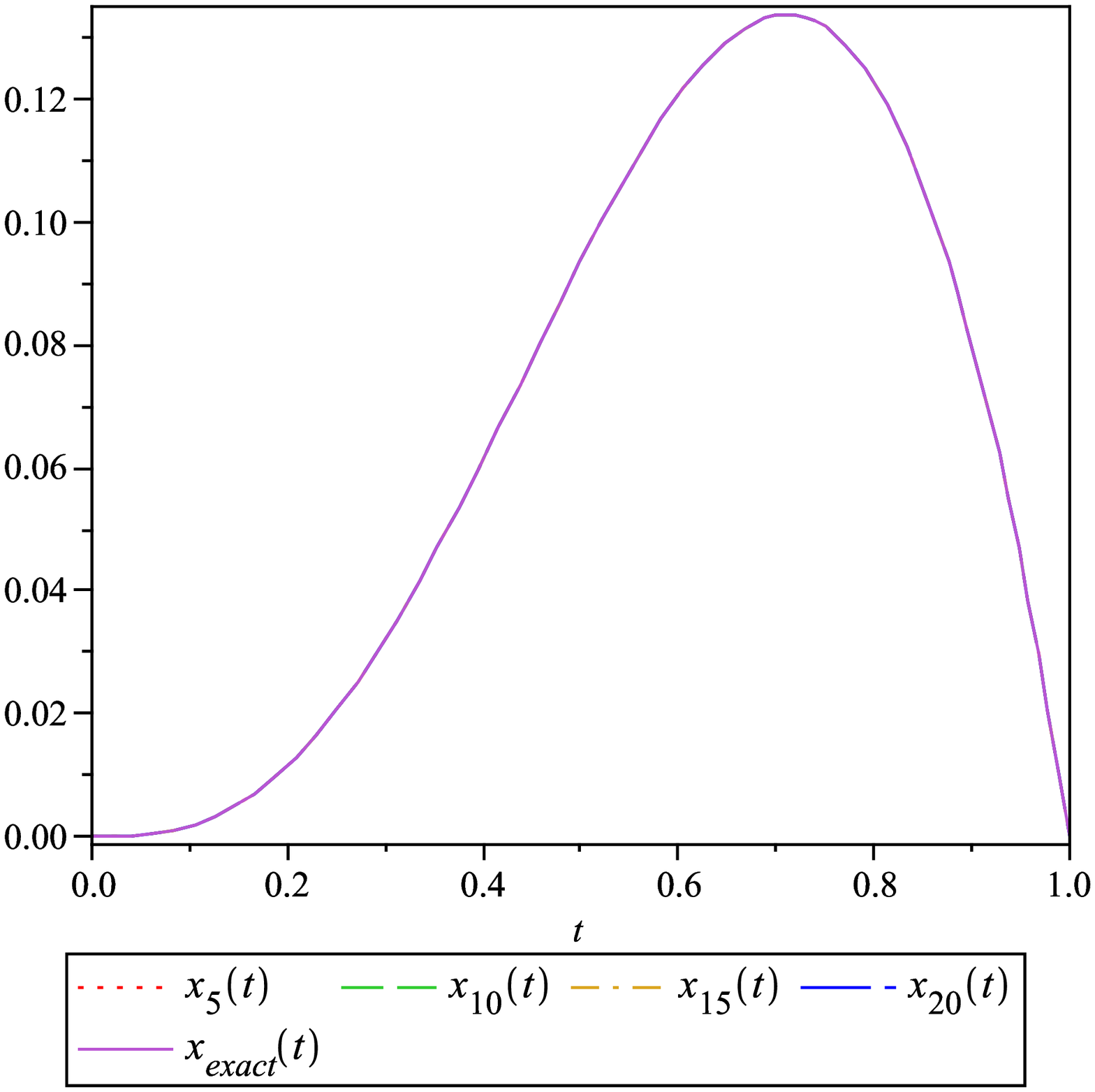}\includegraphics[scale=0.30]{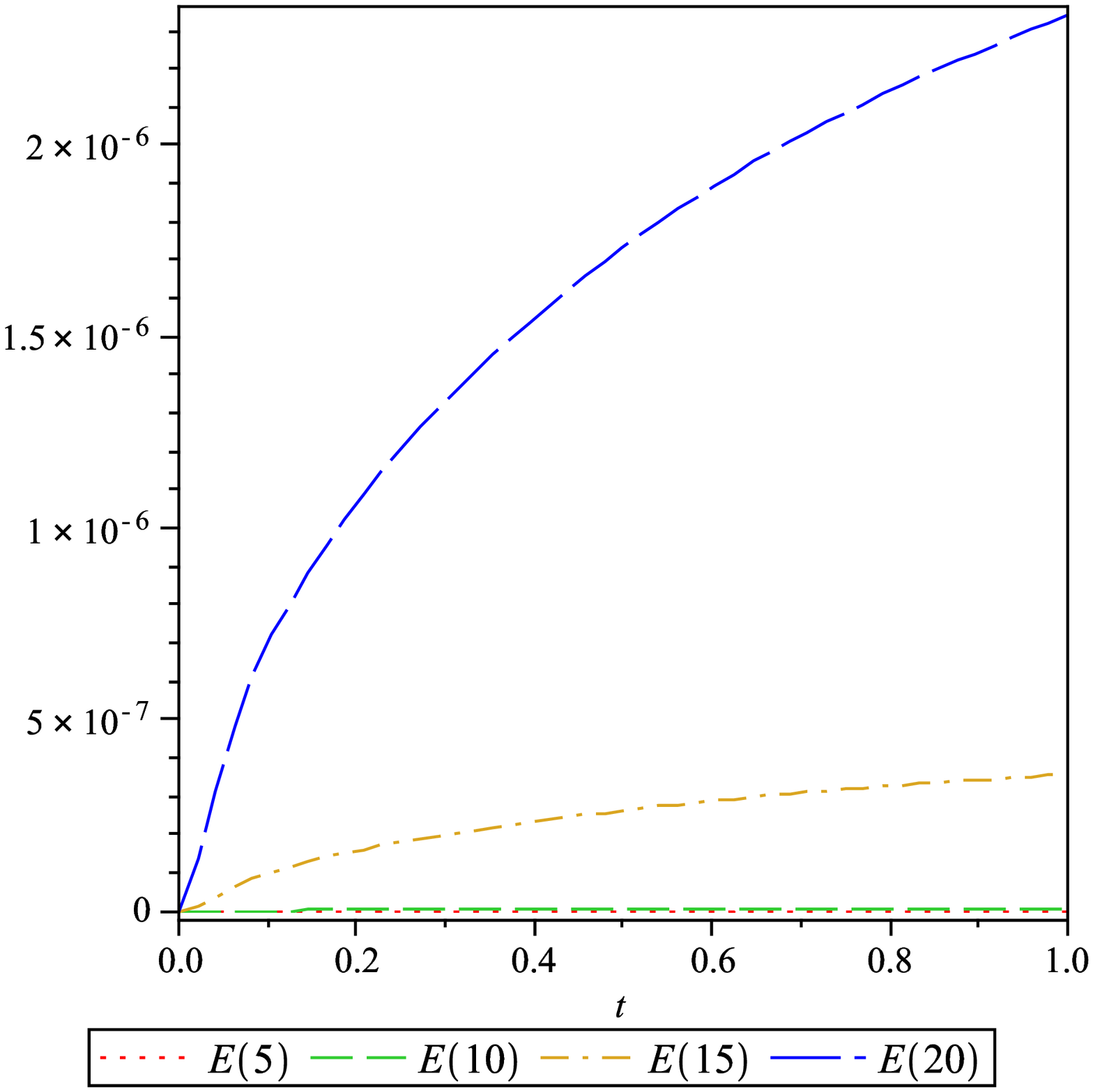}
\caption{Example~\ref{FDEex2}: $x_{exact}(t)$ versus $x_n(t)=B_n(x;t)$ (left) and
$E(n)=|x_{exact}(t)-x_n(t)|$ (right), $n=5,10,15,20$.}\label{FigFD2}
\end{figure}
It is worthwhile to note that the best numerical solution is obtained for $n=5$
because the exact solution is a polynomial of degree five.
Fig.~\ref{FigFD2} (right) shows that $E(n)$ grows for $n>5$.
\end{example}

\begin{example}
\label{FDEex33}
Consider the fractional oscillation equation
$({}_0^CD_{t}^{\alpha}x)(t)+x(t)=te^{-t}$
subject to $x(0)=0$ and $x^{(1)}(0)=0$.
The exact solution is known to be
$x_{exact}(t)=\int_0^tG(t-x)xe^{-x}\,dx$
with $G(t)=t^{\alpha-1}E_{\alpha,\alpha}(-t^\alpha)$ \cite{LinLiu:07:Fra}.
In Fig.~\ref{FigFD3} we compare, for several values of $\alpha$,
the exact solution with the approximation obtained by our method with $n = 15$
(left), and the case $\alpha = \frac{3}{2}$
for various values of $n$ (right).
\begin{figure}
\center
\includegraphics[scale=0.30]{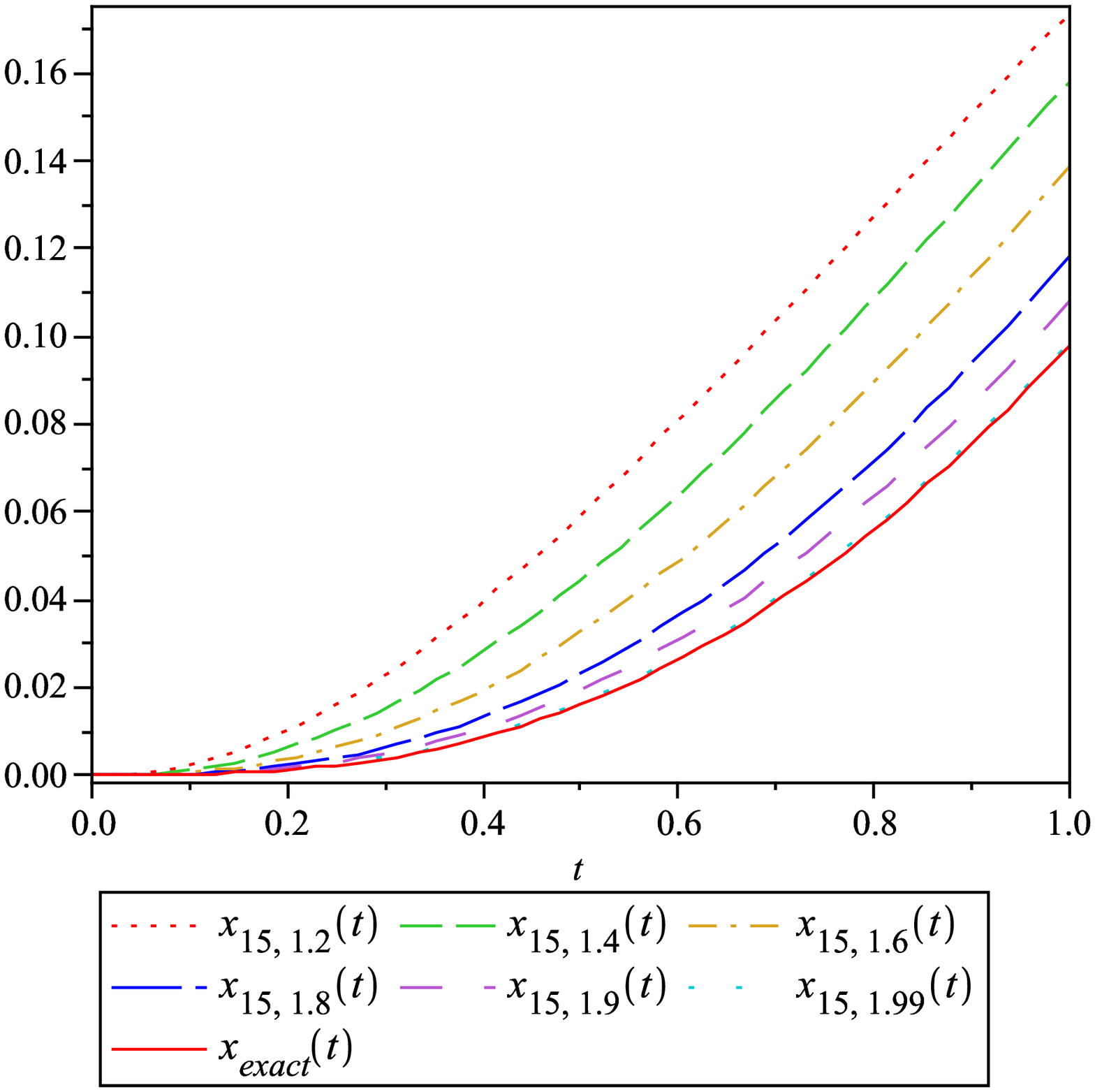}\includegraphics[scale=0.30]{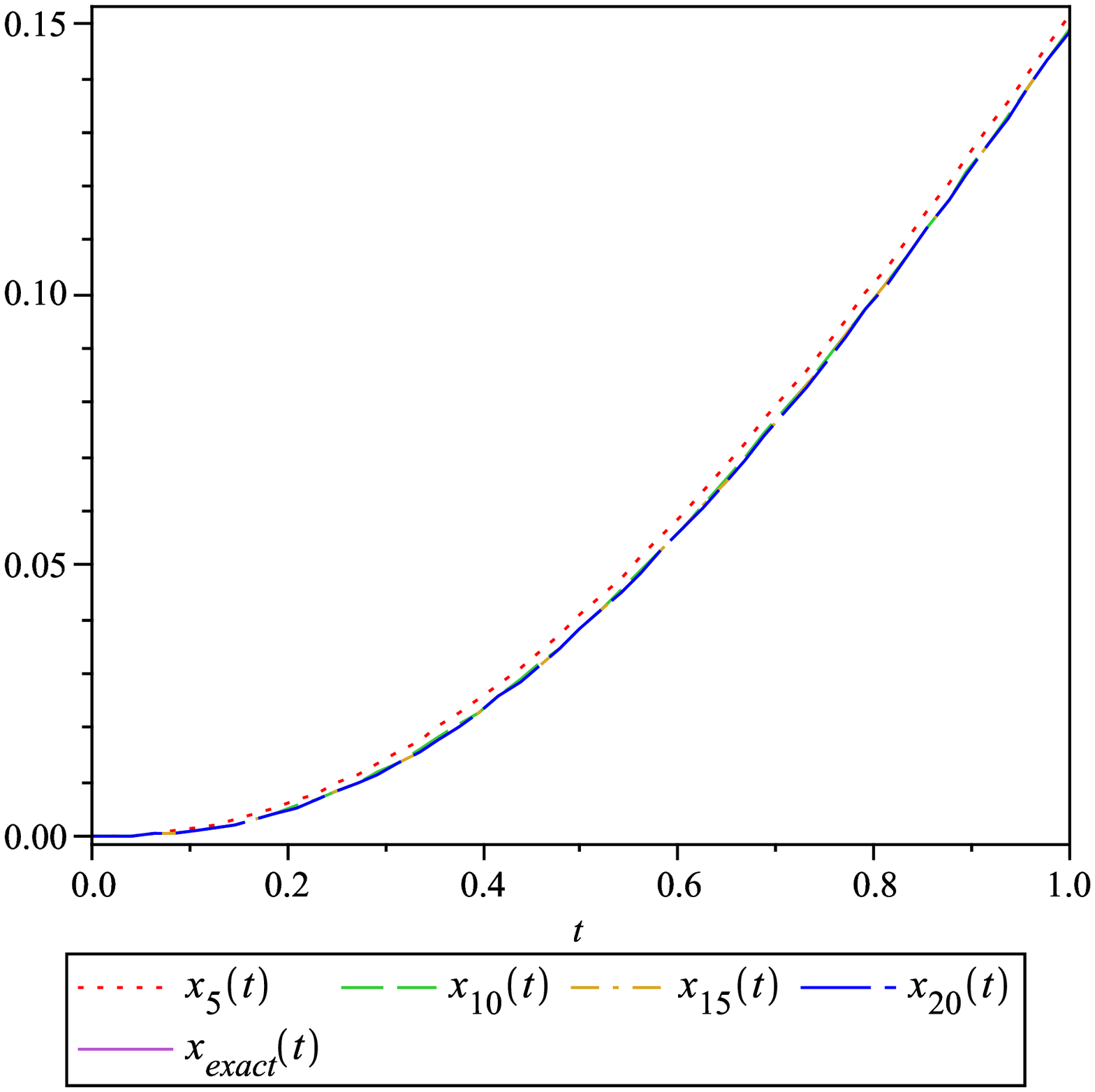}
\caption{Example~\ref{FDEex33}: numerical solution $x_{n,\alpha}(t)$ with $n = 15$
versus the exact solution $x_{exact}(t)$ (left), and
$x_{n}(t) = x_{n,\alpha}(t)$ with $\alpha = \frac{3}{2}$
versus the exact solution $x_{exact}(t)$ (right).}\label{FigFD3}
\end{figure}
\end{example}

\begin{example}
\label{FDEex44}
As our last example, we consider the following
fractional nonlinear differential equation
borrowed from \cite{Kia:The:10}:
$({}_0^CD_{t}^{0.28}x)(t)+(t-0.5)\sin(x(t))=0.8t^3$
subject to a nonzero initial condition $x(0)=x_0$.
The exact solution to this problem is unknown. Let
$z(t)=x(t)-x_0$. Then the problem
is transformed into
$({}_0^CD_{t}^{0.28}z)(t)+(t-0.5)\sin(z(t)+x_0)=0.8t^3$
subject to the zero initial condition $z(0)=0$.
In Fig.~\ref{FigFD4} we show the numerical solutions
($n=10$ and $n = 20$) for several
values of the initial condition:
$x_0=1.2$, $x_0=1.3$, $x_0=1.4$, $x_0=1.5$, and $x_0=1.6$.
\begin{figure}
\center
\includegraphics[scale=0.30]{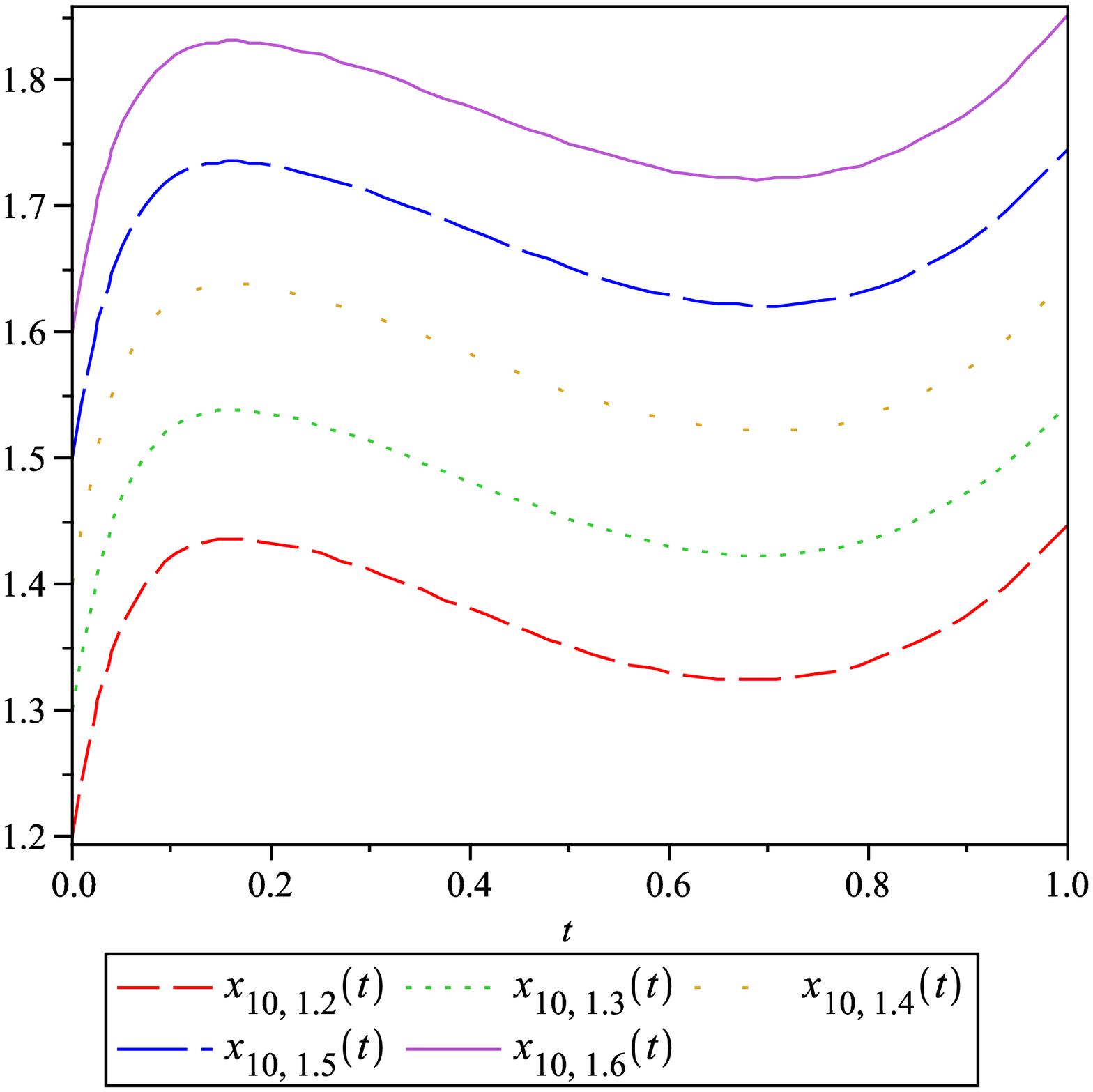}\includegraphics[scale=0.30]{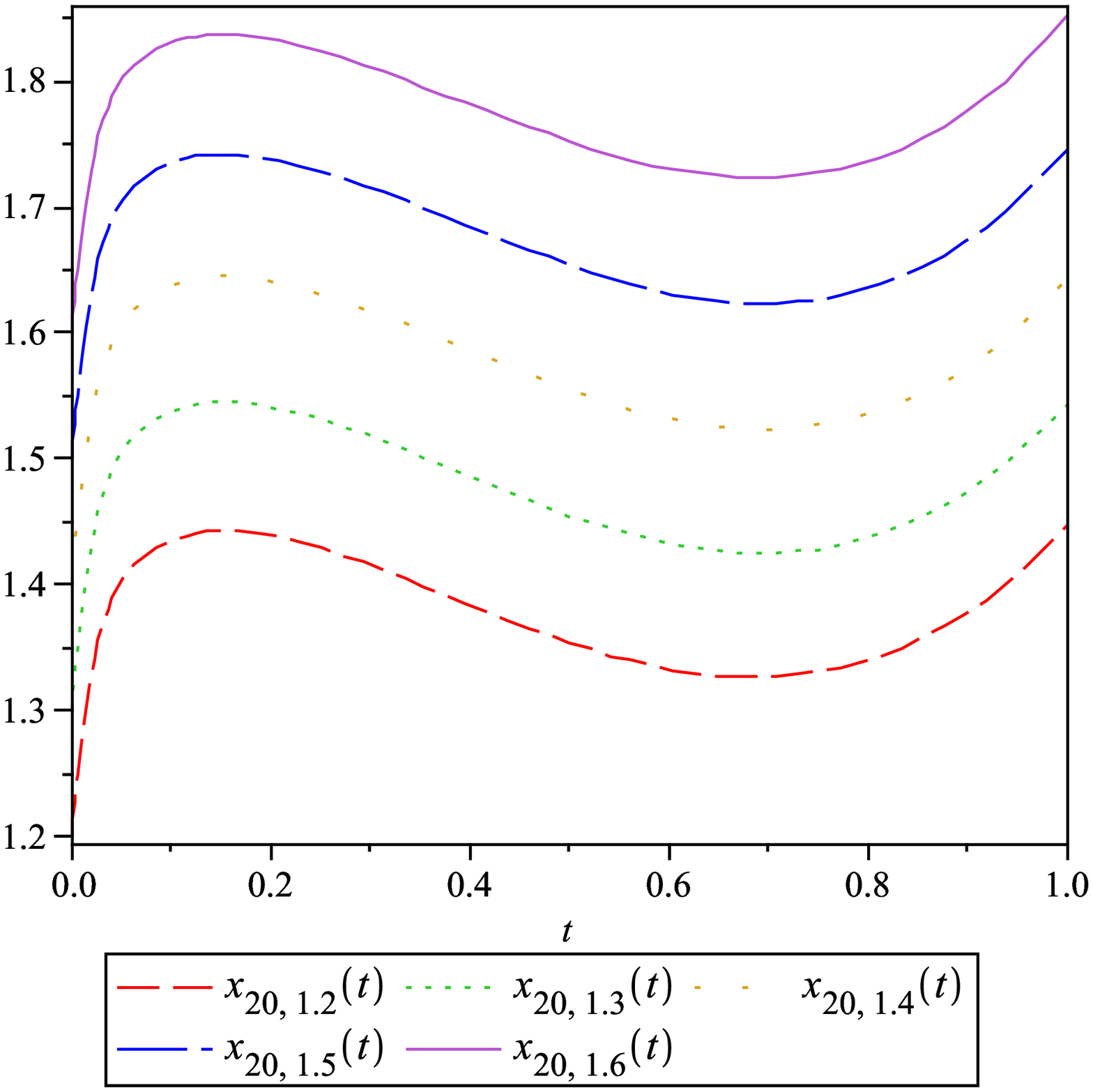}
\caption{Example~\ref{FDEex44}: numerical solutions for $n=10$ (left)
and $n = 20$ (right), $x_0=1.2$, $x_0=1.3$, $x_0=1.4$, $x_0=1.5$, $x_0=1.6$.}\label{FigFD4}
\end{figure}
\end{example}


\section*{Acknowledgements}

This work was partially supported by project PEst-C/MAT/UI4106/2011
with COMPETE number FCOMP-01-0124-FEDER-022690.



\label{pagefin}


\begin{thebibliography}{10}

\bibitem{book:Benchohra}
S. Abbas, M. Benchohra\ and\ G. M. N'Gu\'er\'ekata,
{\it Topics in fractional differential equations},
Developments in Mathematics, 27, Springer, New York, 2012.

\bibitem{Ali_Rost:Bal:JVC}
M. Alipour, D. Rostamy\ and\ D. Baleanu,
Solving multi-dimensional fractional optimal control problems
with inequality constraint by Bernstein polynomials operational matrices,
Journal of Vibration and Control, in press.
DOI:10.1177/1077546312458308

\bibitem{AloNooNaz:09:Hom}
A. K. Alomaria, M. S. M. Noorani, R. Nazar\ and\ C. P. Li,
Homotopy analysis method for solving fractional Lorenz systems,
Commun. Nonlinear Sci. Numer. Simul. {\bf 15} (2010), no.~7, 1864--1872.

\bibitem{HosMil:09:An}
H. Aminikhah\ and\ M. Hemmatnezhad,
An efficient method for quadratic Riccati differential equation,
Commun. Nonlinear Sci. Numer. Simul. {\bf 15} (2010), no.~4, 835--839.

\bibitem{AytIbr:07:Sol}
A. Arikoglu\ and\ I. Ozkol,
Solution of fractional differential equations by using differential transform method,
Chaos Solitons Fractals {\bf 34} (2007), no.~5, 1473--1481.

\bibitem{Dav}
P. J. Davis,
{\it Interpolation and approximation},
Blaisdell Publishing Co. Ginn and Co. New York-Toronto-London, 1963.

\bibitem{Delb:Elliott:94}
D. Delbourgo\ and\ D. Elliott,
On the approximate evaluation of Hadamard finite-part integrals,
IMA J. Numer. Anal. {\bf 14} (1994), no.~4, 485--500.

\bibitem{Kia:The:10}
K. Diethelm,
{\it The analysis of fractional differential equations},
Lecture Notes in Mathematics, 2004, Springer, Berlin, 2010.

\bibitem{EdwForSim:02:The}
J. T. Edwards, N. J. Ford\ and\ A. C. Simpson,
The numerical solution of linear multi-term
fractional differential equations; systems of equations,
J. Comput. Appl. Math. {\bf 148} (2002), no.~2, 401--418.

\bibitem{El-El-El:05:Num}
A. E. M. El-Mesiry, A. M. A. El-Sayed\ and\ H. A. A. El-Saka,
Numerical methods for multi-term fractional (arbitrary)
orders differential equations,
Appl. Math. Comput. {\bf 160} (2005), no.~3, 683--699.

\bibitem{Elliott:93}
D. Elliott,
An asymptotic analysis of two algorithms for certain Hadamard finite-part integrals,
IMA J. Numer. Anal. {\bf 13} (1993), no.~3, 445--462.

\bibitem{ForCon:06:Com}
N. J. Ford\ and\ J. A. Connolly,
Comparison of numerical methods for fractional differential equations,
Commun. Pure Appl. Anal. {\bf 5} (2006), no.~2, 289--306.

\bibitem{Gar:09:On}
R. Garrappa,
On some explicit Adams multistep methods for fractional differential equations,
J. Comput. Appl. Math. {\bf 229} (2009), no.~2, 392--399.

\bibitem{HasAbdMom:09:Hom}
I. Hashim, O. Abdulaziz\ and\ S. Momani,
Homotopy analysis method for fractional IVPs,
Commun. Nonlinear Sci. Numer. Simul. {\bf 14} (2009), no.~3, 674--684.

\bibitem{HuLuoLu:08:Ana}
Y. Hu, Y. Luo\ and\ Z. Lu,
Analytical solution of the linear fractional differential
equation by Adomian decomposition method,
J. Comput. Appl. Math. {\bf 215} (2008), no.~1, 220--229.

\bibitem{KilSriTru:06:The}
A. A. Kilbas, H. M. Srivastava\ and\ J. J. Trujillo,
{\it Theory and applications of fractional differential equations},
North-Holland Mathematics Studies, 204, Elsevier, Amsterdam, 2006.

\bibitem{Yua:09:Sol}
Y. Li,
Solving a nonlinear fractional differential equation using Chebyshev wavelets,
Commun. Nonlinear Sci. Numer. Simul. {\bf 15} (2010), no.~9, 2284--2292.

\bibitem{LinLiu:07:Fra}
R. Lin\ and\ F. Liu,
Fractional high order methods for the nonlinear
fractional ordinary differential equation,
Nonlinear Anal. {\bf 66} (2007), no.~4, 856--869.

\bibitem{Mohammed:12}
O. H. Mohammed\ and\ S. A. Altaie,
Approximate solution of fractional integro-differential
equations by using Bernstein polynomials,
Eng. \& Tech. Journal {\bf 30} (2012), no.~8, 1362--1373.

\bibitem{phi}
G. M. Phillips,
{\it Interpolation and approximation by polynomials},
CMS Books in Mathematics/Ouvrages de Math\'ematiques de la SMC, 14,
Springer, New York, 2003.

\bibitem{Pod:99:Fra(b)}
I. Podlubny,
{\it Fractional differential equations},
Mathematics in Science and Engineering, 198,
Academic Press, San Diego, CA, 1999.

\bibitem{MyID:225}
S. Pooseh, R. Almeida\ and\ D. F. M. Torres,
Approximation of fractional integrals by means of derivatives,
Comput. Math. Appl. {\bf 64} (2012), no.~10, 3090--3100.
{\tt arXiv:1201.5224}

\bibitem{MyID:258}
S. Pooseh, R. Almeida\ and\ D. F. M. Torres,
Discrete direct methods in the fractional calculus of variations,
Comput. Math. Appl., in press.
DOI:10.1016/j.camwa.2013.01.045
{\tt arXiv:1205.4843}

\bibitem{MyID:215}
S. Pooseh, H. S. Rodrigues\ and\ D. F. M. Torres,
Fractional derivatives in Dengue epidemics,
AIP Conf. Proc. 1389 (2011), no.~1, 739--742.
{\tt arXiv:1108.1683}

\bibitem{Rostamy:12}
D. Rostamy\ and\ K. Karimi,
Bernstein polynomials for solving fractional heat- and wave-like equations,
Fract. Calc. Appl. Anal. {\bf 15} (2012), no.~4, 556--571.

\bibitem{SaaDeh:09:Fra}
A. Saadatmandi\ and\ M. Dehghan,
A new operational matrix for solving fractional-order differential equations,
Comput. Math. Appl. {\bf 59} (2010), no.~3, 1326--1336.

\bibitem{SchGau:06:Ona}
A. Schmidt\ and\ L. Gaul,
On a critique of a numerical scheme for calculation
of fractionally damped dynamical systems,
Mech. Res. Comm. {\bf 33} (2006), no.~1, 99–-107.

\bibitem{TakHir:09:Qua}
H. Sugiura\ and\ T. Hasegawa,
Quadrature rule for Abel's equations: uniformly approximating fractional derivatives,
J. Comput. Appl. Math. {\bf 223} (2009), no.~1, 459--468.

\bibitem{Xu:09:Ana}
H. Xu,
Analytical approximations for a population growth model with fractional order,
Commun. Nonlinear Sci. Numer. Simul. {\bf 14} (2009), no.~5, 1978--1983.

\end{thebibliography}
\end{document}